\documentclass[mnsc,sglanonrev]{informs4}
\RequirePackage{tgtermes}
\RequirePackage{newtxtext}
\RequirePackage{newtxmath}
\RequirePackage{bm}
\RequirePackage{endnotes}

\OneAndAHalfSpacedXI 

\usepackage{algorithm}
\usepackage{algpseudocode}
\usepackage{tikz}

\usepackage{natbib}
 \bibpunct[, ]{(}{)}{,}{a}{}{,}%
 \def\bibfont{\small}%
\EquationsNumberedThrough    

\TheoremsNumberedThrough     
\ECRepeatTheorems  %

\MANUSCRIPTNO{}

\usepackage{mathtools}
\usepackage[T1]{fontenc}
\usepackage{fix-cm}
\usepackage{enumitem}

\def\E{{\mathbb E}}

\def\R{{\mathbb R}}

\def\ca{{\mathcal A}}
\def\cb{{\mathcal B}}

\def\ci{{\mathcal I}}

\def\cm{{\mathcal M}}

\def\co{{\mathcal O}}

\def\cs{{\mathcal S}}

\def\cx{{\mathcal X}}

\usepackage[linktocpage,colorlinks,linkcolor=blue,anchorcolor=blue,citecolor=blue,urlcolor=blue,pagebackref]{hyperref}
\hypersetup{}
\renewcommand*{\backref}[1]{\ifx#1\relax \else ~ \fi}

\begin{document}


\RUNAUTHOR{Chen and Zhao}



\RUNTITLE{Policy Optimization}

\TITLE{Benign Nonconvex Landscape for Policy Optimization: Infinite-Horizon Discounted MDPs with General State and Action Spaces}

\ARTICLEAUTHORS{%
\AUTHOR{Xin Chen}
\AFF{Naveen Jindal School of Management, University of Texas at Dallas, \EMAIL{xin.chen@utdallas.edu}}

\AUTHOR{Minda Zhao}
\AFF{H. Milton Stewart School of Industrial and Systems Engineering, Georgia Tech, \EMAIL{mindazhao@gatech.edu}}
} 

\ABSTRACT{We study the optimization landscape for infinite-horizon discounted Markov decision processes (MDPs) with general state and action spaces under structured stationary policy classes. A general weighted policy-iteration approach to establishing global convergence guarantees for policy gradient methods requires closure under weighted policy improvement at every policy, a property that may fail even when the policy class contains an optimal policy. To address this issue, we propose weaker conditions that guarantee the absence of suboptimal stationary points and establish the Polyak--{\L}ojasiewicz--Kurdyka (P{\L}K) condition for the policy gradient objective with a finite concentrability coefficient. We also establish the P{\L}K condition from a policy-improvement bound that holds at every state, without a concentrability assumption. Our general results encompass settings covered by the earlier framework when the common standing assumptions hold for the same policy class and parameter domain. We further verify our proposed conditions for two operations models: inventory systems with Markov-modulated demand and stochastic cash-balance problems. For both models, the Bellman equation yields approximate convexity of the $Q$-value functions in the action variable, with deviations controlled by the first-order stationarity measure. These estimates establish exponent-one and, under additional curvature assumptions, exponent-two P{\L}K conditions, which, together with Lipschitz continuity of the policy gradient, imply an $\mathcal{O}(1/\epsilon)$ iteration complexity and linear convergence, respectively, for projected gradient descent using exact policy gradients. To the best of our knowledge, we provide the first non-asymptotic convergence rates for solving infinite-horizon discounted inventory systems with Markov-modulated demand and stochastic cash-balance problems using policy gradient methods.}





\maketitle

\section{Introduction}\label{sec: introduction}
Reinforcement learning (RL) has found applications in a wide range of domains, including game-playing AI \citep{mnih2015human,silver2016mastering}, the training of large language models \citep{ouyang2022training, guo2025deepseek}, robotics \citep{kober2013reinforcement,hwangbo2019learning}, healthcare \citep{komorowski2018artificial}, and operations management \citep{gijsbrechts2022can,alvo2023deep}. Policy gradient methods constitute an important class of RL algorithms that directly optimize a parameterized policy by applying first-order methods to its expected cumulative cost \citep{sutton1999policy,silver2014deterministic}. This approach is particularly attractive for Markov Decision Processes (MDPs) with continuous state and action spaces, where solving the Bellman equation over all admissible policies can be difficult, while structural analysis often identifies a low-dimensional policy class containing an optimal policy. Classic examples include linear policies for linear quadratic regulator (LQR) problems and base-stock policies for inventory models. Parameterizing such a policy class reduces the original control problem to a finite-dimensional optimization problem. However, the resulting objective is generally nonconvex \citep{agarwal2021theory}, so first-order stationarity alone does not guarantee global optimality.

The central question is whether the structure of an MDP can yield global optimality guarantees despite this nonconvexity. \citet{bhandari2024global} connect policy gradient methods with weighted policy iteration. For each current policy, the weighted policy-iteration objective evaluates candidate policies using the current policy's value function and discounted occupancy measure. Their exact-closure framework requires that, for every current policy in the class, the parameterized class contains a minimizer of the weighted policy-iteration objective over all admissible policies, and the objective, viewed as a function of the candidate policy parameters, has no suboptimal stationary points. The closure requirement can fail even when the class contains an optimal policy for the MDP. We demonstrate this failure in an inventory model: the base-stock class contains an optimal policy for the MDP, but no candidate base-stock policy minimizes the weighted policy-iteration objective associated with a particular current base-stock policy.

Our first result requires the two structural conditions only when the current parameter is a stationary point of the policy gradient objective. Under the stated regularity and exploration assumptions, these conditions suffice to exclude suboptimal stationary points. For quantitative guarantees, we bound two differences in the weighted policy-iteration objective: the improvement obtained by optimizing over the parameterized class, and the additional improvement obtained by allowing all admissible policies. When both differences are bounded by a power of the first-order stationarity measure and the effective concentrability coefficient is finite, the policy gradient objective satisfies the Polyak--{\L}ojasiewicz--Kurdyka (P{\L}K) condition. This condition bounds the objective gap in terms of first-order stationarity and, under smoothness, yields non-asymptotic convergence guarantees for projected gradient descent.

A separate result establishes the P{\L}K condition without assuming a finite effective concentrability coefficient. It requires a policy-improvement bound that holds at every state, rather than only after averaging under the current policy's discounted occupancy measure. If the weighted policy-iteration objective attains its minimum over the parameterized policy class for every current policy in that class, the pointwise condition implies the two quantitative conditions described above, but the converse need not hold. A two-state MDP satisfies the conditions of Theorems~\ref{thm: general-mdp-no-spurious} and~\ref{thm: general-mdp-PLK} while violating the pointwise condition, so the pointwise result does not replace the first two theorems. Conversely, another two-state example satisfies the pointwise condition but has an infinite effective concentrability coefficient, so the two routes are complementary.

We apply the pointwise result to inventory systems with Markov-modulated demand under state-dependent base-stock policies and stochastic cash-balance problems under two-sided base-stock policies. For both models, we establish approximate convexity of the $Q$-value functions in the action variable, with deviations controlled by the first-order stationarity measure. These estimates yield the required pointwise bounds and establish the P{\L}K condition with exponent one under the baseline assumptions and exponent two under an additional demand-density lower bound and positive curvature. We also prove that the policy gradient objectives have Lipschitz continuous gradients under the baseline assumptions. Consequently, projected gradient descent using exact policy gradients and a suitable constant step size attains an $\epsilon$-optimal policy in $\co(1/\epsilon)$ iterations in the first case and $\co(\log(1/\epsilon))$ iterations in the second. The application bounds do not require a finite effective concentrability coefficient, although they retain their dependence on the model-specific constants.

\subsection{Highlights of Contributions}\label{subsec: contributions}
Our work contributes to optimization, operations, and reinforcement learning.

\begin{henumerate}
    \item From an \emph{optimization} perspective, convergence guarantees for algorithms often assume benign landscape conditions \citep{attouch2013convergence,bento2025convergence,lewis2025complexity}, whereas establishing these conditions for a given nonconvex problem requires separate analysis. We complement the literature by establishing the P{\L}K condition for the policy gradient objective in the inventory system with Markov-modulated demands and the stochastic cash-balance problem. Together, these results expand the set of concrete nonconvex problems for which we rigorously establish benign landscape conditions. 

    \item From an \emph{operations} perspective, to the best of our knowledge, we provide the first non-asymptotic convergence guarantees for policy gradient methods applied to the infinite-horizon discounted inventory system with Markov-modulated demand and stochastic cash-balance problem. Prior work analyzes these models through dynamic programming and characterizes their optimal policy structures \citep{song1993inventory,eppen1969cash}, but does not provide finite-time convergence guarantees for directly optimizing the policy parameters.

    \item From a \emph{reinforcement learning} perspective, we establish general landscape results for policy gradient objectives in infinite-horizon discounted MDPs with general state and action spaces. Our result rules out suboptimal stationary points of the policy gradient objective under weighted policy-iteration conditions imposed only at its stationary points, rather than for every current policy in the class as required by the exact-closure framework. We further establish the P{\L}K condition when the effective concentrability coefficient is finite and both the within-class policy-improvement gap and the closure error are bounded by a power of the first-order stationarity measure. A complementary policy-improvement bound that holds at every state yields the P{\L}K condition without a concentrability assumption. Under the common standing assumptions, the exact-closure framework's conditions for global optimality and gradient dominance imply our corresponding structural conditions for the same policy class and parameter domain. Our results therefore recover the corresponding earlier guarantees while also applying to the infinite-horizon discounted inventory and cash-balance models studied here. Unlike the finite-horizon analysis of \citet{chen2024landscape}, our stationary policies use a common parameter vector across all periods. Because gradient contributions from different periods may cancel, the finite-horizon landscape guarantees do not automatically extend to this parameterization.
\end{henumerate}

\subsection{Literature Review}\label{subsec: literature} This study relates to three streams of literature: (i) nonconvex landscape conditions, (ii) global convergence of policy gradient methods, and (iii) solution methods for operations models with structured policies. 

\subsubsection*{Nonconvex Landscape Conditions.} A large body of optimization literature studies benign landscape conditions that guarantee global convergence of first-order methods despite nonconvexity. The Polyak--{\L}ojasiewicz (P{\L}) and Polyak--{\L}ojasiewicz--Kurdyka (P{\L}K) conditions relate function values to first-order information and support convergence analyses for gradient, proximal, and more general descent methods \citep{polyak1963gradient,lojasiewicz1963topological,kurdyka1998gradients,attouch2013convergence,karimi2016linear,bento2025convergence,lewis2025complexity}. Much of this literature assumes a landscape condition and studies its algorithmic consequences. 

A complementary line of work establishes those nonconvex landscape conditions for concrete problem classes, including LQR policy optimization \citep{fazel2018global}, finite-horizon policy optimization \citep{chen2024landscape}, and joint arrival- and service-rate control in queueing systems \citep{chen2025hidden}. We contribute to this line by proving that the policy gradient objectives for the infinite-horizon inventory and cash-balance problems considered here satisfy the P{\L}K condition.

\subsubsection*{Global Convergence of Policy Gradient Methods.} For MDPs with finite state and action spaces, global convergence guarantees have been established for policy gradient and natural policy gradient methods \citep{agarwal2021theory,cen2022fast,bhandari2021linear}, as well as policy mirror descent and its variants \citep{lan2023policy,lan2023block,li2024homotopic,li2025policy}. Related guarantees are available for concave utility objectives and constrained formulations based on state-action occupancy measures \citep{zhang2020variational,ying2025policy}, and for special control models such as LQR under linear policies \citep{fazel2018global}. For MDPs with general state and action spaces, \citet{ju2026policy} develop policy mirror descent and policy dual averaging methods without requiring explicit policy parameterization. Our focus is on the optimization landscape over fixed structured policy classes parameterized by finite-dimensional vectors.

Most closely related, \citet{bhandari2024global} provide structural conditions for global optimality and gradient dominance through weighted policy iteration. They also study approximate closure with a fixed error bound and obtain a corresponding bound on the suboptimality of stationary points. Our qualitative result imposes the structural conditions only at stationary points of the policy gradient objective. Our quantitative result controls both policy-improvement differences by the stationarity measure, so both vanish at stationary points. Using the same effective concentrability coefficient permits a direct comparison of the quantitative conditions under the common regularity and parameter-domain assumptions. The pointwise result provides a separate sufficient condition without assuming that this coefficient is finite.

\citet{chen2024landscape} establish P{\L}K conditions and convergence guarantees for finite-horizon MDPs with general state and action spaces, including inventory and cash-balance models. They consider non-stationary policies and use sequential decomposition inequalities in their landscape analysis. In contrast, we study stationary policies with parameters shared across periods, so gradient contributions from different periods may cancel, and the finite-horizon landscape guarantees do not automatically extend to our setting. For the inventory and cash-balance models, we establish approximate convexity of the $Q$-value functions in the action variable, with deviations controlled by the first-order stationarity measure. This yields the pointwise policy-improvement bounds needed to establish the P{\L}K condition for these models.

\subsubsection*{Solution Methods for Operations Models with Structured Policies.} Classical dynamic-programming studies characterize state-dependent base-stock policies for inventory systems with Markov-modulated demand \citep{song1993inventory}, two-sided threshold policies for stochastic cash-balance problems \citep{eppen1969cash,feinberg2007optimality}, and dual-threshold policies for energy-storage management \citep{harsha2015optimal}. These studies use Bellman equations and structural arguments to characterize optimal policies. Our work addresses a different question: whether direct optimization over the corresponding policy parameters has a globally benign landscape and admits finite-time convergence guarantees. 

A related stream develops optimization and learning methods for operations models. For inventory systems, \citet{glasserman1995sensitivity} develop simulation-based derivative estimators for base-stock levels, while \citet{kunnumkal2008using} and \citet{huh2014online} analyze stochastic-approximation and biased-gradient methods for related formulations. Sampling-based approximation schemes have also been developed for capacitated inventory models \citep{cheung2019sampling}. More recently, \citet{li2026convergence} establish convergence and inference results for stream stochastic gradient descent, including a lost-sales inventory model with a long-run average-cost objective. Deep RL has also been applied to complex inventory systems and networks \citep{gijsbrechts2022can,alvo2023deep}. We focus on the optimization landscape of the policy gradient objective for infinite-horizon discounted inventory and cash-balance models and establish iteration complexity for exact policy gradients.

\subsection{Organization}
The remainder of the paper is organized as follows. Section~\ref{sec: problem formulation} introduces the discounted MDP formulation and policy optimization. Section~\ref{sec: PLK condition} defines the P{\L}K condition and gives convergence rates for projected gradient descent. Section~\ref{sec:limitations-global-pi} presents an inventory counterexample to closure under weighted policy improvement. Section~\ref{sec: landscape} develops the landscape results and compares their assumptions. Section~\ref{sec: application} verifies the pointwise condition for the inventory and cash-balance models and derives their convergence guarantees. Appendix~\ref{appendix: landscape} contains the proofs of the general landscape results, and Appendix~\ref{appendix: inventory-core} contains the core inventory analysis. The electronic companion contains the remaining proofs and technical verifications.

\section{Problem Formulation and Preliminaries}\label{sec: problem formulation}
We defer the measurability and integrability assumptions required for a rigorous treatment of general state and action spaces to the end of this section to keep the initial formulation accessible. Consider an infinite-horizon discounted MDP $(\cs, \ca, g, P, \gamma, \rho)$ defined in \citet{puterman2014markov}: the state space $\cs \subseteq \R^m$; the action space $\ca \coloneqq \bigcup_{s\in\cs} \ca_s$, where $\ca_s \subseteq \R^n$ is the nonempty set of feasible actions for state $s \in \cs$; the cost function $g: \cs \times \ca \to \R$, where $g(s, a)$ is the one-period cost after taking action $a$ in state $s$; the transition kernel $P$ specifies the conditional distribution $P(\cdot | s, a)$ of the next state given a feasible state-action pair $(s, a)$; the discount factor $\gamma \in (0, 1)$, and the initial state distribution $\rho$.

Let $\Pi$ denote the set of admissible measurable stationary policies. A policy $\pi\in\Pi$ maps each state $s\in\cs$ to a feasible action $\pi(s)\in\ca_s$. For a policy $\pi$, the value function starting from state $s$ is
\begin{equation*}
    J_\pi(s) \coloneqq \mathbb E_s^\pi \left[ \sum_{t=0}^{\infty}\gamma^t g(s_t,\pi(s_t)) \right].
\end{equation*}
The expectation $\E_s^\pi$ denotes that we take the expectation over a Markovian sequence $(s_0, s_1,\dots)$, where $s_0 = s$ and $s_{t+1} \sim P(\cdot|s_t, \pi(s_t))$. A policy $\pi^*$ is optimal if it minimizes the total expected cost
\begin{equation*}
    l(\pi) \coloneqq (1-\gamma) \int_\cs J_\pi(s)\rho(ds) = (1-\gamma) \mathbb E_{s\sim\rho}^\pi \left[ \sum_{t=0}^{\infty}\gamma^t g(s_t,\pi(s_t)) \right].
\end{equation*}
For every Borel measurable set $\cm\subseteq\cs$, we define the discounted state-occupancy measure $\eta_\pi$ by
\begin{equation*}
    \eta_\pi(\cm) \coloneqq (1-\gamma)\sum_{t=0}^{\infty}\gamma^t \mathbb P_\rho^\pi(s_t\in\cm).
\end{equation*}
Here, $\mathbb P_\rho^\pi(s_t\in\cdot)$ denotes the distribution of the state at time $t$ with initial distribution $\rho$ and policy $\pi$. The discounted state-occupancy measure provides the discounted fraction of time the Markovian system spends in $\cm$. Hence, the total expected cost allows for an equivalent expression:
\begin{equation*}
    l(\pi) = \int_\cs g(s, \pi(s)) \eta_\pi(ds).
\end{equation*}

\subsection{Bellman Equations}
We define several functions and operators used in our main theory. For a measurable function $J$, define the Bellman operators
\begin{equation*}
    \begin{aligned}
        (T_\pi J)(s) &\coloneqq g(s,\pi(s)) + \gamma\int_\cs J(s')P(ds'|s,\pi(s)),\\
        (TJ)(s) &\coloneqq \inf_{a\in\ca_s} \left\{ g(s,a) + \gamma\int_\cs J(s')P(ds'|s,a) \right\}.
    \end{aligned}
\end{equation*}
For any policy $\pi \in \Pi$, define the state-action value function (or the $Q$-value function)
\begin{equation*}
    Q_\pi(s,a) \coloneqq g(s,a) + \gamma\int_\cs J_\pi(s')P(ds'|s,a).
\end{equation*}
By definition, for any policy $\pi, \pi' \in \Pi$, we have
\begin{equation*}
    Q_\pi(s, \pi'(s)) = (T_{\pi'}J_\pi)(s), \qquad \inf_{a\in\ca_s} Q_\pi(s,a) = (TJ_\pi)(s).
\end{equation*}
Additionally, the value function $J_{\pi}$ and $Q$-value function $Q_{\pi}$ satisfy the following Bellman equations:
\begin{equation*}
    \left\{    
    \begin{aligned}
        J_\pi(s) & = Q_{\pi} \left( s, \pi(s) \right),\\
        Q_{\pi}(s, a) & = g(s, a) + \gamma\int_\cs J_\pi(s')P(ds'|s,a).
    \end{aligned}
    \right.
\end{equation*}

\subsection{Policy Gradient Methods}
Policy gradient methods use first-order optimization algorithms to minimize the total expected cost $l(\pi)$. In this paper, we focus on policies parameterized by finite-dimensional vectors $\theta$ \citep{sutton1999policy}. We denote each such policy by $\pi_\theta$, where $\theta$ belongs to a nonempty closed convex set $\Theta\subseteq\mathbb R^d$. By working in the affine hull of $\Theta$ and using orthonormal coordinates if necessary, we may assume without loss of generality that $\Theta$ has nonempty interior. The corresponding parameterized policy class is $\Pi_\Theta=\{\pi_\theta:\theta\in\Theta\}\subseteq\Pi$.

Given a parameterized policy $\pi_\theta$, we overload the notation of the total expected cost by $l(\theta) \coloneqq l(\pi_\theta)$, called the policy gradient objective function. The corresponding policy optimization problem is $\min_{\theta\in\Theta}l(\theta)$. We define a policy $\pi_{\theta}$ to be $\epsilon$-optimal if $\theta$ is an $\epsilon$-optimal solution of $\min_{\theta\in\Theta}l(\theta)$. Let $\theta^*$ denote one of the minimizers of $\min_{\theta\in\Theta} l(\theta)$ and $\pi_{\theta^*}$ as the corresponding policy. Throughout the paper, we assume that the parameterized policy class $\Pi_\Theta$ contains the optimal policy $\pi^*$. This often occurs when the optimal policy class is known, e.g., affine policies in the Linear Quadratic Regulator (LQR) problem and base-stock policies in the inventory control model.

\subsection{Measurability and Integrability Assumptions}
The following assumption collects the standard technical conditions needed to ensure that the controlled process, Bellman expressions, and functions used below are well defined. We state these conditions separately so that the subsequent analysis can focus on the optimization landscape. The measurability assumption is standard in dynamic programming with general Borel state and action spaces; see \citet{blackwell1965discounted,bertsekas1978stochastic}, \citet[Section~3.3]{hernandez2012discrete}, and \citet[Section~6.2.5]{puterman2014markov}. Throughout, we refer to measurability with respect to the relevant Borel $\sigma$-algebras.

\begin{assumption}\label{assumption: general-mdp-technical}
    Assume that both of the following assumptions hold.
    \begin{henumerate}
        \item \label{assumption: general-mdp-measurability}
        \textbf{(Measurability.)} The state space $\cs$ and action space $\ca$ are Borel subsets of Euclidean spaces, and the set of feasible state-action pairs $\Gamma\coloneqq\left\{(s, a)\in\cs\times\ca:a\in\ca_s\right\}$ is measurable. The one-period cost function $g:\Gamma\to\R$ is measurable. The transition kernel $P$ is a stochastic kernel on $\cs$ given $\Gamma$: for every $(s,a)\in\Gamma$, $P(\cdot\mid s,a)$ is a probability measure on $\cs$, and, for every measurable set $\cm\subseteq\cs$, the mapping $(s,a)\mapsto P(\cm\mid s,a)$ is measurable on $\Gamma$. For every $\theta\in\Theta$, the value function $J_{\pi_\theta}$ is real-valued and measurable and satisfies $J_{\pi_\theta} = T_{\pi_\theta}J_{\pi_\theta}$. Moreover, $Q_{\pi_\theta}$ is real-valued and measurable on $\Gamma$, $TJ_{\pi_\theta}$ is real-valued and measurable on $\cs$, and there exists a policy $\pi'\in\Pi$ such that
        \begin{equation*}
            Q_{\pi_\theta}(s,\pi'(s)) = \inf_{a\in\ca_s}Q_{\pi_\theta}(s,a) = (TJ_{\pi_\theta})(s), \qquad s\in\cs.
        \end{equation*}
    
        \item \label{assumption: general-mdp-integrability} 
        \textbf{(Integrability.)} For every $\theta,\theta',\bar\theta\in\Theta$, the functions $J_{\pi_\theta}$, $T_{\pi_{\theta'}}J_{\pi_\theta}$, and $TJ_{\pi_\theta}$ are absolutely integrable with respect to $\eta_{\pi_{\bar\theta}}$. In addition, for every $\theta\in\Theta$ and every $\pi'\in\Pi$, the function $T_{\pi'}J_{\pi_\theta}$ is absolutely integrable with respect to $\eta_{\pi_\theta}$.
    \end{henumerate}
\end{assumption}

We maintain Assumption~\ref{assumption: general-mdp-technical} throughout the general landscape analysis in Section~\ref{sec: landscape} and verify it for each application in Section~\ref{sec: application}.

\section{Nonconvex Landscape Conditions and Convergence Rates}\label{sec: PLK condition}
The policy gradient objective function $l(\theta)$ is generally nonconvex, so first-order methods cannot guarantee convergence to global optimal solutions under smoothness alone. Standard results for nonconvex optimization typically provide only first-order stationarity guarantees. If every stationary point is globally optimal, then convergence to stationary points also implies convergence to a globally optimal policy. However, the absence of suboptimal stationary points does not provide a non-asymptotic convergence rate for the objective gap. Such rates require a quantitative characterization of the optimization landscape. This section introduces the landscape conditions used in our analysis, explains their relationships, and states the convergence of projected gradient descent under these nonconvex conditions.

\subsection{P{\L}K Condition}\label{subsec: PLK definition}
In this work, we use the Polyak-{\L}ojasiewicz-Kurdyka (P{\L}K) condition \citep{polyak1963gradient, lojasiewicz1963topological, kurdyka1998gradients, bento2025convergence}, which is well-suited to the following analysis. 

\begin{definition}[P{\L}K Condition]\label{def: PLK condition}
    Consider a nonempty closed convex set $\cx\subseteq\R^n$. Suppose that $f$ is differentiable on an open set containing $\cx$ and attains its minimum over $\cx$. The function $f$ satisfies the P{\L}K condition on $\cx$ if there exist $\mu>0$ and $\alpha\in[1,2]$ such that
    \begin{equation*}
        f(x)-\min_{y\in\cx}f(y)\le\frac{1}{2\mu}\min_{g\in\partial\delta_\cx(x)}\|\nabla f(x)+g\|_2^\alpha,\qquad \forall x\in\cx.
    \end{equation*}
    Here, $\mu$ denotes the P{\L}K constant and $\alpha$ denotes the P{\L}K exponent. Furthermore, $\partial\delta_\cx(x)$ denotes the convex subdifferential of the indicator function $\delta_\cx$ at $x$, which is the normal cone
    \begin{equation*}
        \partial\delta_\cx(x)=N_\cx(x)\coloneqq\left\{g\in\R^n:\langle g,y-x\rangle\le0,\ \forall y\in\cx\right\}.
    \end{equation*}
\end{definition}

The P{\L}K condition ensures that every stationary point is globally optimal. To see this, if $x$ is a stationary point, i.e., $\langle \nabla f(x), x' - x \rangle \ge 0$ for any $x' \in \cx$, then $-\nabla f(x) \in \partial \delta_\cx(x) = N_\cx(x)$, which implies
\begin{equation*}
    f(x) - \min_{x\in\cx} f(x) \le \frac{1}{2\mu} \min_{g \in \partial \delta_\cx(x)} \left \| \nabla f(x) + g \right \|_2^\alpha \le \frac{1}{2\mu} \left \| \nabla f(x) - \nabla f(x) \right \|_2^\alpha = 0.
\end{equation*}

\subsection{Gradient Dominance and Its Relation to P{\L}K Condition}\label{subsec: gradient dominance}
For comparison with existing policy gradient guarantees, we recall the gradient-dominance condition of \citet[Definition~2]{bhandari2024global}.

\begin{definition}[Gradient Dominance]\label{def: gradient dominance}
    Consider a nonempty closed convex set $\cx\subseteq\R^n$. Suppose that $f$ is differentiable on an open set containing $\cx$ and attains its minimum over $\cx$. For constants $c>0$ and $\mu\ge0$, the function $f$ is $(c,\mu)$-gradient dominated over $\cx$ if
    \begin{equation}\label{eq:BR-gradient-dominance}
        \min_{y\in\cx}f(y)\ge f(x)+\min_{y\in\cx}\left\{c\langle\nabla f(x),y-x\rangle+\frac{\mu}{2}\|y-x\|_2^2\right\},\qquad \forall x\in\cx,
    \end{equation}
    where the minimum on the right-hand side is assumed to be attained for every $x\in\cx$.
\end{definition}

Convex functions on bounded domains are $(1,0)$-gradient dominated, while $\mu$-strongly convex functions are $(1,\mu)$-gradient dominated. The following lemma relates gradient dominance to the P{\L}K condition.

\begin{lemma} \label{lemma: gradient dominance and PLK}
    Consider a nonempty closed convex set $\cx\subseteq\R^n$, and suppose that $f$ is differentiable on an open set containing $\cx$. Suppose further that $f$ is $(c,\mu)$-gradient dominated over $\cx$ for some $c>0$ and $\mu\ge0$.
    \begin{henumerate}
        \item \label{condition: GD to PLK two} \citep{chen2024landscape} If $\mu>0$, then $f$ satisfies the P{\L}K condition with exponent $2$ and constant $\mu/c^2$.

        \item \label{condition: GD to PLK one} If $\mu=0$ and $\cx$ has diameter at most $R$ for some $R>0$, namely, $\|x-y\|_2\le R$ for all $x,y\in\cx$, then $f$ satisfies the P{\L}K condition with exponent $1$ and constant $1/(2Rc)$.
    \end{henumerate}
\end{lemma}

Setting $c=1$ in Lemma~\ref{lemma: gradient dominance and PLK} shows that $\mu$-strongly convex functions satisfy the P{\L}K condition with exponent $2$ and constant $\mu$, while convex functions on domains of diameter at most $R$ satisfy the P{\L}K condition with exponent $1$ and constant $1/(2R)$. The convex case also follows from \citet[Proposition~2(i)]{fatkhullin2025stochastic} by taking the identity transformation.

\subsection{Convergence Rates for Projected Gradient Descent}\label{subsec: PLK convergence}
Under a P{\L}K condition and Lipschitz continuity of the gradient, the objective values generated by projected gradient descent with an appropriate step size converge to the global minimum at a non-asymptotic rate. We record the convergence rates for the two P{\L}K exponents used in our applications. Under the P{\L}K condition with exponent $1$, Lemma~\ref{lemma: convergence rate}.\ref{alg: PGD 1-KL} establishes an $\co(1/\epsilon)$ iteration complexity. The proof adapts the weak gradient-mapping domination analysis of \citet[Theorem~4]{xiao2022convergence} to the normal-cone formulation. When the P{\L}K condition holds with exponent $2$, Lemma~\ref{lemma: convergence rate}.\ref{alg: PGD 2-KL} establishes a linear convergence rate for the objective gap.
\begin{lemma}\label{lemma: convergence rate}
    Consider an optimization problem $\min_{x\in\cx} f(x)$ over a nonempty closed convex set $\cx$. Suppose that $f$ attains its minimum over $\cx$, is differentiable on an open set containing $\cx$, and has an $L$-Lipschitz continuous gradient on $\cx$ for some $L>0$. Let $f^*=\min_{x\in\cx}f(x)$. Given $x_0\in\cx$, consider projected gradient descent $x_{k+1}=\operatorname{Proj}_\cx(x_k-\frac{1}{L}\nabla f(x_k))$, where $\operatorname{Proj}_\cx$ denotes the projection onto $\cx$.
    \begin{henumerate}
        \item \label{alg: PGD 1-KL} If $f$ satisfies the P{\L}K condition with exponent $1$ and constant $\mu$, then,
        \begin{equation*}
            f(x_k)-f^* \le  \frac{2L/\mu^2+ \left[f(x_0)-f^*\right]}{k+1}, \qquad k \ge 0.
        \end{equation*}
        \item \label{alg: PGD 2-KL} \citep{attouch2013convergence, chen2024landscape} If $f$ satisfies the P{\L}K condition with exponent $2$ and constant $\mu$, then,
        \begin{equation*}
            f(x_{k}) - f^* \le \left(1 - \frac{\mu}{4L + \mu} \right)^k \left[ f(x_0) - f^* \right], \qquad k \ge 0.
        \end{equation*}
    \end{henumerate}
\end{lemma}

\section{Limitations of Closure under Weighted Policy Improvement} \label{sec:limitations-global-pi}
Although gradient dominance, the P{\L}K condition, and several other benign nonconvex conditions guarantee non-asymptotic convergence rates for first-order methods under smoothness assumptions, verifying these conditions for the policy gradient objective remains difficult because of the complexity of MDPs. The exact-closure framework reviewed in Section~\ref{subsec: literature} identifies sufficient conditions under which the policy gradient objective satisfies a gradient dominance condition. We next examine its closure requirement.

Whenever the following integral is finite, define the weighted policy-iteration objective by
\begin{equation*}
    \cb(\pi'\mid\eta,J_\pi) \coloneqq \int_\cs (T_{\pi'}J_\pi)(s) \eta(ds) = \int_\cs Q_\pi(s,\pi'(s)) \eta(ds),
\end{equation*}
where $\eta$ is a probability measure over $\cs$. When $\cs$ is discrete and $\eta(\{s\})>0$ for every $s\in\cs$, a classical policy-iteration update can be equivalently written as
\begin{equation*}
    \pi_{k+1} \in \argmin_{\pi\in\Pi} \cb( \pi \mid \eta, J_{\pi_k} ).
\end{equation*}

We recall the following conditions from \citet[Conditions~1, 2A, and~2B]{bhandari2024global}, retaining their names and numbering and stating them in our notation. Here, PI denotes policy iteration, and stationary points are understood relative to the feasible parameter set $\Theta$.
\begin{hitemize}
    \item \textbf{Condition 1 (Closure Under Policy Improvement).} For every $\theta\in\Theta$, there exists $\theta^+\in\Theta$ such that
    \begin{equation}\label{eq:BR-global-closure}
        \cb(\pi_{\theta^+}\mid\eta_{\pi_\theta},J_{\pi_\theta})=\min_{\pi'\in\Pi}\cb(\pi'\mid\eta_{\pi_\theta},J_{\pi_\theta}).
    \end{equation}

    \item \textbf{Condition 2A (Stationary Points of the Weighted PI Objective).} For every $\theta\in\Theta$, every stationary point of the function $\theta'\mapsto\cb(\pi_{\theta'}\mid\eta_{\pi_\theta},J_{\pi_\theta})$ is a global minimizer over $\Theta$.

    \item \textbf{Condition 2B (Gradient Dominance of the Weighted PI Objective).} There exist constants $c>0$ and $\mu\ge0$ such that, for every $\theta\in\Theta$, the function $\theta'\mapsto\cb(\pi_{\theta'}\mid\eta_{\pi_\theta},J_{\pi_\theta})$ is $(c,\mu)$-gradient dominated over $\Theta$.
\end{hitemize}

Condition~1 requires that the weighted policy-improvement problem admit an optimal solution within the parameterized class. Conditions~2A and~2B describe the landscape of this problem, with Condition~2B strengthening Condition~2A. Under the corresponding regularity assumptions, Conditions~1 and~2A ensure that the policy gradient objective has no suboptimal stationary points; Conditions~1 and~2B additionally yield gradient dominance when the effective concentrability coefficient is finite. We next show that Condition~1 can fail for an inventory model even though the base-stock policy class contains an optimal policy.

\subsection{Inventory Model} \label{subsec: single-state-inventory}

We consider an infinite-horizon discounted inventory model with discount factor $\gamma\in(0,1)$. The state $x_t$ is the inventory level before ordering, and the action $y_t$ is the post-order inventory level. Let $B>0$ denote the warehouse capacity. The state and action spaces are $\cs=\ca=(-\infty,B]$, and the feasible action set at state $x\in\cs$ is $\ca_x=[x,B]$. Let $\{D_t\}_{t\ge0}$ be an i.i.d.\ nonnegative demand process with distribution $F_D$ and finite mean. After choosing $y_t\in\ca_{x_t}$, the next inventory level is $x_{t+1}=y_t-D_t$. For unit holding cost $h\ge0$ and unit backlogging cost $b\ge0$, the one-period expected cost given the post-order inventory level is
\begin{equation*}
    C(y) \coloneqq \mathbb E \left[ h(y-D)^+ + b(D-y)^+ \right].
\end{equation*}
Hence, the one-period cost is $g(x,y)=C(y)$ for $y\in\ca_x$. We use $\rho$ to denote the initial distribution of $x_0$, and let $\eta_\pi$ denote the discounted state-occupancy measure induced by $\pi$ and $\rho$.

For a stationary policy $\pi \in \Pi$, feasibility requires $x \le \pi(x) \le B$ for $x \in \cs$. The value function from Section~\ref{sec: problem formulation} takes the form
\begin{equation*}
    J_\pi(x)=\mathbb E_x^\pi \left[\sum_{t=0}^{\infty}\gamma^t C(y_t)\right], \quad y_t = \pi(x_t), \quad x_{t+1} = y_t - D_t, \quad x_0 = x.
\end{equation*}
The goal is to identify an ordering policy that minimizes $J_\pi(x)$ for all $x \in \cs$. We focus on the stationary base-stock policy class $\pi_\theta(x)=x\vee\theta$ with $\theta \in \Theta = [0, B]$, which contains an optimal stationary policy \citep{bertsekas1995dynamic}. Then, the policy gradient objective function is
\begin{equation*}
    l(\theta) = l(\pi_\theta) = (1-\gamma) \int_\cs J_{\pi_\theta}(x) \rho(dx) = (1-\gamma) \mathbb E_{x_0\sim\rho}^{\pi_\theta} \left[ \sum_{t=0}^{\infty}\gamma^t C(x_t \vee \theta) \right].
\end{equation*}

\subsection{Counterexample to Closure under Policy Improvement} \label{subsec: global-closure-counterexample}
We use the inventory model in Section~\ref{subsec: single-state-inventory} to show that the closure under policy improvement \eqref{eq:BR-global-closure} can fail for the stationary base-stock policy class. Let $\gamma=1/2$, $h=1/3$, $b=1$, and $B = 4$. Let $D\sim\operatorname{Unif}[2,3]$, and let the initial inventory level satisfy $x_0\sim\operatorname{Unif}[-1,4]$, independent of the demand process. Consider the base-stock policy $\pi(x)=x\vee1$. For every possible post-order inventory level $y\le4$, define
\begin{equation*}
    G(y) \coloneqq C(y) + \frac{1}{2} \mathbb E \left[ J_\pi(y-D) \right].
\end{equation*}
The Bellman equation gives $J_\pi(x) = G(x\vee1)$. Therefore,
\begin{equation*}
    G(y) = C(y) + \frac{1}{2} \int_{y-3}^{y-2} G(u\vee1) \,du, \qquad y\le4.
\end{equation*}
Then, we have
\begin{equation*}
    \cb(\pi_{\theta'}\mid\eta_\pi,J_\pi)=\int_{\cs}G(x\vee\theta')\,\eta_\pi(dx), \quad \text{and} \quad \cb(\pi'\mid\eta_\pi,J_\pi) = \int_{\cs}G(\pi'(x))\,\eta_\pi(dx).
\end{equation*}

For every measurable set $\cm\subseteq\cs$, the definition of the discounted state-occupancy measure gives $\eta_\pi(\cm) \ge (1-\gamma)\rho(\cm)$. Since $\rho$ has a strictly positive density on $[-1,4]$, every nondegenerate subinterval of $[-1,4]$ has positive $\eta_\pi$-measure. Moreover, because a measurable policy attaining the pointwise Bellman minimum exists, a policy $\pi'\in\Pi$ minimizes the unrestricted weighted policy-iteration objective only if $G(\pi'(x)) = \min_{y\in[x,4]} G(y)$ for $\eta_\pi$-almost every $x\in\cs$.

First, $G(1)=\mathbb E[D-1]+G(1)/2$. Therefore, $G(1) = 3$ and $G(y)=4-y$ for $y\le2$. For $2\le y\le3$, $y-D\le1$ almost surely, so
\begin{equation*} 
    G(y)=\frac{1}{6}(y-2)^2+\frac{1}{2}(3-y)^2+\frac{3}{2}=\frac{13}{8}+\frac{2}{3}\left(y-\frac{11}{4}\right)^2. 
\end{equation*}
Thus, $G$ is uniquely minimized over $[2,3]$ at $y_0=11/4$, with $G(y_0)=13/8$. Moreover, $G(y)\ge2>13/8$ for $y\le2$. For $3\le y\le4$, we have $y-D\in[0,2]$. Using $G(1)=3$ and $G(u)=4-u$ for $1\le u\le2$,
\begin{equation*}
    G(y) = \frac{1}{3}\left(y-\frac{5}{2}\right)+\frac{1}{2}\left\{\int_{y-3}^{1}G(1)\,du+\int_1^{y-2}(4-u)\,du\right\} =-\frac{1}{4}y^2+\frac{11}{6}y-\frac{19}{12}.
\end{equation*}
Since this expression is concave on $[3,4]$, its minimum on that interval is $\min\{G(3),G(4)\}=5/3>13/8$. Combining the above cases, we obtain $G(y) \ge 13/8$ for $y\le4$ with equality if and only if $y = 11/4$. Therefore, $11/4$ is the unique minimizer of $G$ over the set of possible post-order inventory levels $\ca=(-\infty,4]$.

We now argue by contradiction. Suppose that a base-stock policy $\pi_{\theta'}(x)=x\vee\theta'$ minimizes the weighted policy-iteration objective. For every $x\in(1,11/4)$, the unique minimizer of $G$ over $[x,4]$ is $11/4$. If $\theta'\ne11/4$, then $x\vee\theta'\ne11/4$ throughout this interval, which has positive $\eta_\pi$-measure. Hence, weighted optimality forces $\theta'=11/4$. However, the expression for $G$ on $[3,4]$ implies
\begin{equation*}
    G(y)-G(4)=\frac{(4-y)(3y-10)}{12}, \qquad 3\le y\le4.
\end{equation*}
Therefore, $G(4)<G(x)$ for every $x\in(10/3,4)$. On this interval, the candidate base-stock policy with $\theta'=11/4$ selects $y=x$, whereas the feasible action $y=4$ yields a strictly smaller value of $G$. Since this interval has positive $\eta_\pi$-measure, the candidate cannot minimize the weighted policy-iteration objective. Hence, closure under policy improvement (Condition~1) fails.

As we see later, this instance satisfies the first three parts of Assumption~\ref{assumption: inventory}, so Theorem~\ref{thm: markov-demand-PLK}\ref{markov-demand-no-suboptimal-stationary} establishes an exponent-one PŁK condition despite the failure of Condition 1.

\section{Optimization Landscape}\label{sec: landscape}
The inventory counterexample in Section~\ref{sec:limitations-global-pi} shows that Closure Under Policy Improvement (Condition~1) can be restrictive. Motivated by this observation, we develop new sufficient conditions for establishing benign nonconvexity of the policy gradient objective. These conditions cover the inventory and cash-balance models considered later. We start with the following regularity conditions.

\begin{assumption} \label{assumption: general-mdp-regularity}
The following regularity conditions hold.
    \begin{henumerate}
        \item \label{assumption: optimal policy} The parameterized policy class $\Pi_\Theta$ contains the optimal policy $\pi^*$ with corresponding parameter $\theta^*$.

        \item \label{assumption: differentiablity} \textbf{(Differentiability.)} For each $\theta\in\Theta$, the mapping $(\theta',\bar\theta) \mapsto \cb( \pi_{\theta'} \mid \eta_{\pi_{\bar\theta}}, J_{\pi_\theta})$ is well-defined on $\Theta\times\Theta$ and admits a continuously differentiable extension to an open set containing $\Theta\times\Theta$.
    \end{henumerate}
\end{assumption}

Assumption~\ref{assumption: general-mdp-regularity}.\ref{assumption: optimal policy} is standard in structured MDPs: it requires the parameterized policy class to contain an optimal policy. For instance, affine policies are optimal for the LQR problem, and base-stock policies are optimal for the inventory model. Second, Assumption~\ref{assumption: general-mdp-regularity}.\ref{assumption: differentiablity} guarantees that the policy gradient objective $l(\theta)$ is differentiable. Under Assumption~\ref{assumption: general-mdp-technical} and Assumption~\ref{assumption: general-mdp-regularity}.\ref{assumption: differentiablity}, the proof of the Policy Gradient Theorem in \citet[Lemma~6]{bhandari2024global} applies directly. Consequently, $l$ is continuously differentiable and
\begin{equation} \label{eq: general-mdp-policy-gradient}
    \nabla l(\theta) = \left. \nabla_{\theta'} \cb(\pi_{\theta'} \mid \eta_{\pi_\theta}, J_{\pi_\theta}) \right|_{\theta'=\theta}.
\end{equation}

We use the following notion of stationarity. A point $\theta\in\Theta$ is a stationary point of a differentiable function $f:\Theta\to\mathbb R$ over $\Theta$ if
\begin{equation*}
    \langle \nabla f(\theta),\theta'-\theta\rangle\ge 0, \qquad \forall \theta'\in\Theta.
\end{equation*}
For the policy gradient objective $l$, define the first-order stationarity measure
\begin{equation*}
    \mathcal R(\theta) \coloneqq \min_{g\in N_\Theta(\theta)} \|\nabla l(\theta)+g\|_2.
\end{equation*}
Then $\mathcal R(\theta)=0$ if and only if $\theta$ is a stationary point of $l$ over $\Theta$.

\subsection{No Spurious Stationary Points}
Our first result provides new conditions under which the policy gradient objective $l$ has no suboptimal stationary points over $\Theta$. Conditions~1 and~2A impose Closure Under Policy Improvement and Stationary Points of the Weighted PI Objective, respectively, for every current policy in the parameterized class. The following theorem requires these properties only when the current policy parameter $\theta$ is a stationary point of $l$ over $\Theta$.

\begin{theorem}[No Suboptimal Stationary Points] \label{thm: general-mdp-no-spurious}
    Suppose that Assumption~\ref{assumption: general-mdp-regularity} holds and $\eta_{\pi_{\theta^*}}$ is absolutely continuous with respect to $\rho$. For every stationary point $\theta$ of $l$ over $\Theta$, if the following two conditions hold
    \begin{henumerate}
        \item \label{condition: pointwise-closure-pi} The parameterized policy class is closed under weighted policy improvement at $\theta$: there exists $\theta^+\in\Theta$ such that
        \begin{equation*}
            \cb(\pi_{\theta^+}\mid\eta_{\pi_\theta},J_{\pi_\theta})=\min_{\pi'\in\Pi}\cb(\pi'\mid\eta_{\pi_\theta},J_{\pi_\theta}).
        \end{equation*}
        
        \item \label{condition: pointwise-no-spurious-pi} 
        The weighted policy-iteration objective $\theta' \mapsto \cb(\pi_{\theta'} \mid \eta_{\pi_\theta}, J_{\pi_\theta})$ has no suboptimal stationary points over $\Theta$. That is, if $\bar\theta\in\Theta$ is a stationary point of $\theta' \mapsto \cb(\pi_{\theta'} \mid \eta_{\pi_\theta}, J_{\pi_\theta})$, then
        \begin{equation*}
            \cb(\pi_{\bar\theta} \mid \eta_{\pi_\theta}, J_{\pi_\theta}) = \min_{\theta'\in\Theta} \cb(\pi_{\theta'} \mid \eta_{\pi_\theta}, J_{\pi_\theta}).
        \end{equation*}
        
    \end{henumerate}
    Then, $\theta$ is globally optimal over $\Theta$:
    \begin{equation*}
        l(\theta)\le l(\theta')\qquad \forall \theta'\in\Theta.
    \end{equation*}
\end{theorem}

\subsection{P{\L}K Condition}
Next, we establish the P{\L}K condition for policy optimization by providing alternatives to Conditions~1 and~2B. To relate weighted policy improvement to the infinite-horizon objective gap, we use the effective concentrability coefficient of \citet{bhandari2024global}, defined below.

\begin{definition}[Effective Concentrability Coefficient]\label{def: effective-concentrability}
    Define $\kappa_\rho$ for the policy class $\Pi_\Theta$ and initial distribution $\rho$ as the smallest nonnegative constant satisfying
    \begin{equation}\label{ineq: effective-concentrability}
        \int_\cs\left[J_{\pi_\theta}(s)-J_{\pi_{\theta^*}}(s)\right]\rho(ds)\le\frac{\kappa_\rho}{1-\gamma}\int_\cs\left[J_{\pi_\theta}(s)-(TJ_{\pi_\theta})(s)\right]\rho(ds),\qquad \forall\theta\in\Theta.
    \end{equation}
    If no finite constant satisfies this inequality, set $\kappa_\rho=\infty$.
\end{definition}
A finite $\kappa_\rho$ ensures that the improvement from optimizing the current action and following $\pi_\theta$ thereafter, averaged under $\rho$, controls the infinite-horizon objective gap uniformly over $\Pi_\Theta$. This relation allows bounds on weighted policy improvement to yield a P{\L}K condition for $l$.

\begin{theorem}[P\L{}K Condition]\label{thm: general-mdp-PLK}
    Suppose Assumption~\ref{assumption: general-mdp-regularity} holds and $\kappa_\rho<\infty$. Assume that, for every $\theta\in\Theta$, the weighted policy-iteration objective $\theta'\mapsto\cb(\pi_{\theta'}\mid\eta_{\pi_\theta},J_{\pi_\theta})$ attains its minimum over $\Theta$. Assume further that there exist constants $C_{\rm PI},C_{\rm cl}\ge0$ and $\alpha\in[1,2]$ such that the following conditions hold for every $\theta\in\Theta$.
    \begin{henumerate}
        \item The parameterized policy class satisfies the approximate closure condition:
        \begin{equation}\label{condition: general-mdp-approx-closure-pi}
            \min_{\theta'\in\Theta}\cb(\pi_{\theta'}\mid\eta_{\pi_\theta},J_{\pi_\theta})-\min_{\pi'\in\Pi}\cb(\pi'\mid\eta_{\pi_\theta},J_{\pi_\theta})\le C_{\rm cl}\mathcal R(\theta)^\alpha.
        \end{equation}

        \item The weighted policy-iteration objective satisfies:
        \begin{equation}\label{condition: general-mdp-approx-no-spurious-pi}
            \cb(\pi_\theta\mid\eta_{\pi_\theta},J_{\pi_\theta})-\min_{\theta'\in\Theta}\cb(\pi_{\theta'}\mid\eta_{\pi_\theta},J_{\pi_\theta})\le C_{\rm PI}\mathcal R(\theta)^\alpha.
        \end{equation}
    \end{henumerate}
    Then $l$ satisfies the P{\L}K condition on $\Theta$ with exponent $\alpha$:
    \begin{equation*}
        l(\theta)-l(\theta^*)\le\frac{\kappa_\rho(C_{\rm PI}+C_{\rm cl})}{1-\gamma}\mathcal R(\theta)^\alpha, \qquad \forall\theta\in\Theta.
    \end{equation*}
\end{theorem}

A sufficient condition for $\kappa_\rho<\infty$ is a finite \textit{distribution-mismatch coefficient} $\left\| d\eta_{\pi_{\theta^*}} / d\rho \right\|_\infty$. Specifically, under the standing assumptions, if $\eta_{\pi_{\theta^*}}$ is absolutely continuous with respect to $\rho$, then
\begin{equation}\label{ineq: effective-concentrability-density-ratio}
    \kappa_\rho\le\left\|\frac{d\eta_{\pi_{\theta^*}}}{d\rho}\right\|_\infty.
\end{equation}
This bound follows from the performance difference lemma; see also \citet[Theorem~4(b)]{bhandari2024global}. The \emph{distribution-mismatch coefficient} measures how well the initial distribution covers the states visited by the optimal policy. In contrast, $\kappa_\rho$ only needs to satisfy~\eqref{ineq: effective-concentrability} for the value functions generated by $\Pi_\Theta$. Consequently, $\kappa_\rho$ can be finite even when the distribution-mismatch coefficient is infinite. When the latter is finite, substituting~\eqref{ineq: effective-concentrability-density-ratio} into Theorem~\ref{thm: general-mdp-PLK} gives
\begin{equation*}
    l(\theta)-l(\theta^*)\le\frac{C_{\rm PI}+C_{\rm cl}}{1-\gamma}\left\|\frac{d\eta_{\pi_{\theta^*}}}{d\rho}\right\|_\infty\mathcal R(\theta)^\alpha, \qquad \forall\theta\in\Theta.
\end{equation*}

Under Lipschitz continuity of $\nabla l$, Theorem~\ref{thm: general-mdp-PLK} and Lemma~\ref{lemma: convergence rate} yield an $\co(1/\epsilon)$ iteration complexity for projected gradient descent using exact policy gradients when $\alpha=1$, and linear convergence when $\alpha=2$. For comparison, \citet[Lemma~3]{bhandari2024global} state an $\co(1/\epsilon^2)$ iteration bound under $(c,0)$-gradient dominance on a bounded domain. Combining Lemmas~\ref{lemma: gradient dominance and PLK} and~\ref{lemma: convergence rate} yields the sharper $\co(1/\epsilon)$ bound under the same assumptions. This improvement follows from a sharper convergence analysis of projected gradient descent. For $\alpha=2$, our result also yields linear convergence for the objective gap. Their Lemma~3 states the linear rate for an unconstrained domain, with constrained extensions discussed separately.

We compare the conditions of Theorem~\ref{thm: general-mdp-PLK} with Conditions~1 and~2B under the common standing assumptions and for the same policy class and parameter domain.
\begin{hitemize}
    \item Condition~1 (Closure Under Policy Improvement) is the special case of \eqref{condition: general-mdp-approx-closure-pi} with $C_{\rm cl}=0$. Our condition allows a positive closure error away from stationary points of $l$ over $\Theta$, provided that this error is bounded by $C_{\rm cl}\mathcal R(\theta)^\alpha$. At stationary points of $l$, this bound requires exact closure.

    \item To compare Condition~2B (Gradient Dominance of the Weighted PI Objective) with \eqref{condition: general-mdp-approx-no-spurious-pi}, write
    \begin{equation*}
        \Phi_\theta(\theta') \coloneqq\cb(\pi_{\theta'}\mid\eta_{\pi_\theta},J_{\pi_\theta}).
    \end{equation*}
    Here, $\theta$ fixes the current policy, while $\theta'$ is the candidate parameter. By Lemma~\ref{lemma: gradient dominance and PLK}, Condition~2B implies
    \begin{equation*}
        \Phi_\theta(z)-\min_{\theta'\in\Theta}\Phi_\theta(\theta') \le C_{\rm PI}\min_{g\in N_\Theta(z)} \|\nabla\Phi_\theta(z)+g\|_2^\alpha, \qquad \forall\theta,z\in\Theta,
    \end{equation*}
    with $\alpha=2$ and $C_{\rm PI}=c^2/(2\mu)$ when $\mu>0$, or with $\alpha=1$ and $C_{\rm PI}=cR$ when $\mu=0$ and $\Theta$ has diameter at most $R>0$. In comparison, the policy gradient~\eqref{eq: general-mdp-policy-gradient} allows our condition~\eqref{condition: general-mdp-approx-no-spurious-pi} to be written as
    \begin{equation*}
        \Phi_\theta(\theta)-\min_{\theta'\in\Theta}\Phi_\theta(\theta') \le C_{\rm PI}\mathcal R(\theta)^\alpha =C_{\rm PI}\min_{g\in N_\Theta(\theta)} \|\nabla\Phi_\theta(\theta)+g\|_2^\alpha, \qquad \forall\theta\in\Theta.
    \end{equation*}
    Thus, our condition requires the bound only at the candidate $z=\theta$ for each current policy, whereas a global P{\L}K condition for $\Phi_\theta$ requires it at every candidate parameter $z\in\Theta$. 
\end{hitemize}

Consequently, Conditions~1 and~2B imply the two requirements of Theorem~\ref{thm: general-mdp-PLK} in the cases specified above, with $C_{\rm cl}=0$ and the stated values of $C_{\rm PI}$.

\subsection{P{\L}K Condition Without a Concentrability Assumption}\label{subsec: pointwise-PLK}
Theorem~\ref{thm: general-mdp-PLK} assumes $\kappa_\rho<\infty$ in addition to Conditions~\eqref{condition: general-mdp-approx-closure-pi} and~\eqref{condition: general-mdp-approx-no-spurious-pi}. The following theorem establishes the P{\L}K condition without assuming $\kappa_\rho<\infty$. It requires a bound on $J_{\pi_\theta}(s)-(TJ_{\pi_\theta})(s)$ that holds at every state, rather than only on average under the current policy's discounted occupancy measure.

\begin{theorem}\label{thm: general-mdp-pointwise-PLK}
    Suppose Assumptions~\ref{assumption: general-mdp-technical} and~\ref{assumption: general-mdp-regularity} hold. Assume that there exist constants $C>0$ and $\alpha\in[1,2]$ such that
    \begin{equation}\label{condition: general-mdp-pointwise-error}
        J_{\pi_\theta}(s)-(TJ_{\pi_\theta})(s)\le C\mathcal R(\theta)^\alpha, \qquad \forall \theta\in\Theta,\ s\in\cs.
    \end{equation}
    Then $l$ satisfies the P{\L}K condition on $\Theta$ with exponent $\alpha$ and constant $1/(2C)$:
    \begin{equation}\label{ineq: general-mdp-pointwise-PLK}
        l(\theta)-l(\theta^*)\le C\mathcal R(\theta)^\alpha, \qquad \forall \theta\in\Theta.
    \end{equation}
\end{theorem}

At any stationary point $\theta$ of $l$ over $\Theta$, we have $\mathcal R(\theta)=0$. Since $J_{\pi_\theta}\ge TJ_{\pi_\theta}$, Condition~\eqref{condition: general-mdp-pointwise-error} implies the Bellman optimality equation $J_{\pi_\theta}=TJ_{\pi_\theta}$ and hence $l(\theta)=l(\theta^*)$ under the standing assumptions. Furthermore, if the weighted policy-iteration objective attains its minimum over $\Theta$ for every current parameter $\theta$, then Condition~\eqref{condition: general-mdp-pointwise-error} implies Conditions~\eqref{condition: general-mdp-approx-closure-pi} and~\eqref{condition: general-mdp-approx-no-spurious-pi} in Theorem~\ref{thm: general-mdp-PLK}. To see this, fix $\theta\in\Theta$. Since $\pi_\theta\in\Pi_\Theta\subseteq\Pi$,
\begin{equation*}
    \min_{\pi'\in\Pi}\cb(\pi'\mid\eta_{\pi_\theta},J_{\pi_\theta}) \le\min_{\theta'\in\Theta}\cb(\pi_{\theta'}\mid\eta_{\pi_\theta},J_{\pi_\theta}) \le\cb(\pi_\theta\mid\eta_{\pi_\theta},J_{\pi_\theta}).
\end{equation*}
The left-hand sides of Conditions~\eqref{condition: general-mdp-approx-closure-pi} and~\eqref{condition: general-mdp-approx-no-spurious-pi} are the differences between the first two and the last two quantities, respectively. By the Bellman equation and Condition~\eqref{condition: general-mdp-pointwise-error},
\begin{equation*}
    \cb(\pi_\theta\mid\eta_{\pi_\theta},J_{\pi_\theta})-\min_{\pi'\in\Pi}\cb(\pi'\mid\eta_{\pi_\theta},J_{\pi_\theta}) =\int_\cs\left[J_{\pi_\theta}(s)-(TJ_{\pi_\theta})(s)\right]\eta_{\pi_\theta}(ds) \le C\mathcal R(\theta)^\alpha.
\end{equation*}
Thus, Conditions~\eqref{condition: general-mdp-approx-closure-pi} and~\eqref{condition: general-mdp-approx-no-spurious-pi} hold with $C_{\rm cl}=C_{\rm PI}=C$ and the same exponent $\alpha$. However, Conditions~\eqref{condition: general-mdp-approx-closure-pi} and~\eqref{condition: general-mdp-approx-no-spurious-pi} do not, in general, imply the pointwise condition~\eqref{condition: general-mdp-pointwise-error}. The following two-state MDP satisfies Conditions~1 and~2B, as well as all the assumptions of Theorems~\ref{thm: general-mdp-no-spurious} and~\ref{thm: general-mdp-PLK}, but violates~\eqref{condition: general-mdp-pointwise-error}.

\begin{example}\label{example: weighted-without-pointwise}
    Consider the two-state, two-action MDP in Figure~\ref{fig: tabular MDP}, with $\cs=\{0,1\}$, discount factor $\gamma=1/2$, and initial distribution $\rho(\{0\})=1$. Action $0$ incurs cost $0$, whereas action $1$ incurs cost $1$. At state $0$, action $0$ leaves the state unchanged, and action $1$ moves the process to state $1$. State $1$ is absorbing under both actions.
    
    \begin{figure}[htbp]
        \FIGURE
        {\includegraphics[width=0.5\textwidth]{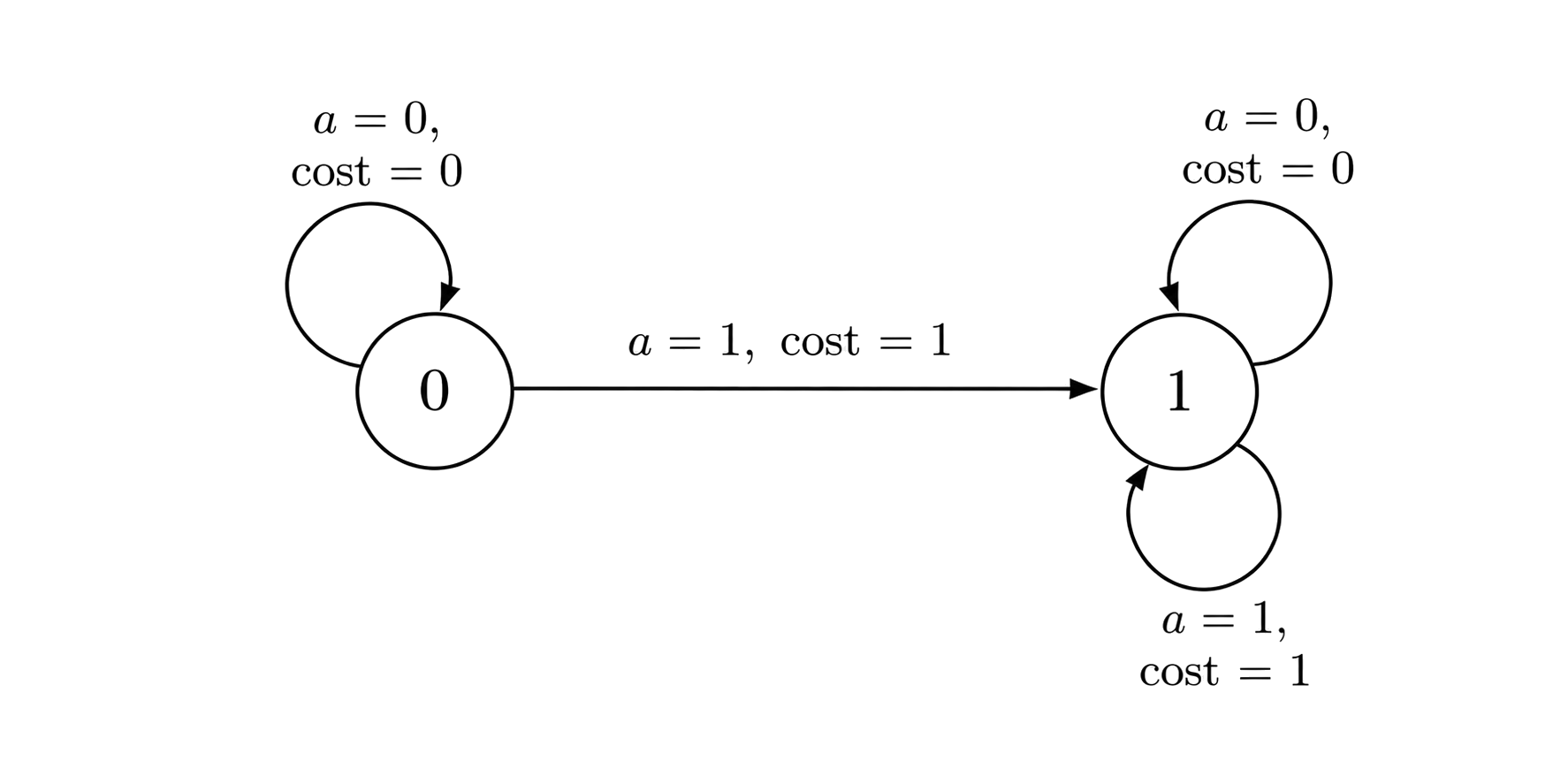}}
        {Two-State, Two-Action MDP in Example~\ref{example: weighted-without-pointwise}\label{fig: tabular MDP}}
        {}
    \end{figure}
    
    To represent stochastic policies in our formulation, let $a\in[0,1]$ denote the probability of selecting action $1$. The feasible action set is then $\ca_s=[0,1]$, and the expected one-period cost and transition probabilities are
    \begin{equation*}
        \begin{aligned}
            g(s,a)=a, \qquad s\in\{0,1\}, \qquad  P(\{0\}\mid0,a) =1-a, \qquad P(\{1\}\mid0,a)=a, \qquad P(\{1\}\mid1,a)=1.
        \end{aligned}
    \end{equation*}
    The policy class $\pi_\theta(s)=\theta_s$, with $\theta=(\theta_0,\theta_1)\in\Theta=[0,1]^2$, therefore represents all stationary stochastic policies of the original MDP. A direct calculation gives
    \begin{equation*}
        \begin{aligned}
            J_{\pi_\theta}(0)&=\frac{2\theta_0(1+\theta_1)}{1+\theta_0}, \qquad J_{\pi_\theta}(1)=2\theta_1, \qquad \eta_{\pi_\theta}(\{0\}) = \frac{1}{1+\theta_0}, \qquad \eta_{\pi_\theta}(\{1\})=\frac{\theta_0}{1+\theta_0},\\
            l(\theta)&=\frac{\theta_0(1+\theta_1)}{1+\theta_0}, \qquad \nabla l(\theta)=\left(\frac{1+\theta_1}{(1+\theta_0)^2},\frac{\theta_0}{1+\theta_0}\right).
        \end{aligned}
    \end{equation*}

    We first verify the assumptions common to Theorems~\ref{thm: general-mdp-no-spurious} and~\ref{thm: general-mdp-PLK}. The policy $\pi_{\theta^*}$ with $\theta^*=(0,0)$ incurs zero cost from either state and is therefore optimal. Since it remains in state $0$ when initialized from $\rho$,
    \begin{equation*}
        \eta_{\pi_{\theta^*}}=\rho, \qquad \left\|\frac{d\eta_{\pi_{\theta^*}}}{d\rho}\right\|_\infty=1.
    \end{equation*}
    Thus, the absolute-continuity assumption of Theorem~\ref{thm: general-mdp-no-spurious} holds. Moreover,~\eqref{ineq: effective-concentrability-density-ratio} gives $\kappa_\rho\le1<\infty$, verifying the effective concentrability assumption of Theorem~\ref{thm: general-mdp-PLK}. The finite state space and bounded costs ensure the required measurability and integrability, and the affine Bellman objectives on $[0,1]$ admit measurable minimizers. Moreover, the displayed expressions for the value functions and occupancy measures, together with the affine dependence on the candidate policy, show that $(\theta',\bar\theta)\mapsto\cb(\pi_{\theta'}\mid\eta_{\pi_{\bar\theta}},J_{\pi_\theta})$ admits a continuously differentiable extension to an open set containing $\Theta\times\Theta$. Hence, Assumptions~\ref{assumption: general-mdp-technical} and~\ref{assumption: general-mdp-regularity} hold.

    Every stationary stochastic policy is determined by its probabilities of selecting action $1$ in states $0$ and $1$. Since $\theta'_0$ and $\theta'_1$ can each take any value in $[0,1]$, optimizing over $\theta'\in\Theta$ is equivalent to optimizing over all policies $\pi'\in\Pi$ in our formulation. Hence, for every $\theta\in\Theta$,
    \begin{equation*}
        \min_{\theta'\in\Theta}\cb(\pi_{\theta'}\mid\eta_{\pi_\theta},J_{\pi_\theta})=\min_{\pi'\in\Pi}\cb(\pi'\mid\eta_{\pi_\theta},J_{\pi_\theta}).
    \end{equation*}
    This equality verifies Condition~1 and hence the closure requirement in Theorem~\ref{thm: general-mdp-no-spurious}. It also verifies~\eqref{condition: general-mdp-approx-closure-pi} with $C_{\rm cl}=0$. For every fixed $\theta\in\Theta$, the weighted policy-iteration objective $\theta'\mapsto\cb(\pi_{\theta'}\mid\eta_{\pi_\theta},J_{\pi_\theta})$ is affine. Therefore, Condition~2A holds, and Condition~2B holds with $c=1$ and $\mu=0$. This verifies the remaining structural requirement of Theorem~\ref{thm: general-mdp-no-spurious}. Furthermore, since $\Theta$ has diameter $\sqrt{2}$, Lemma~\ref{lemma: gradient dominance and PLK} and the policy-gradient identity~\eqref{eq: general-mdp-policy-gradient} give
    \begin{equation*}
        \begin{aligned}
            \cb(\pi_\theta\mid\eta_{\pi_\theta},J_{\pi_\theta})-\min_{\theta'\in\Theta}\cb(\pi_{\theta'}\mid\eta_{\pi_\theta},J_{\pi_\theta}) \le\sqrt{2}\min_{g\in N_\Theta(\theta)}\|\nabla l(\theta)+g\|_2 =\sqrt{2}\mathcal R(\theta).
        \end{aligned}
    \end{equation*}
    Therefore, Condition~\eqref{condition: general-mdp-approx-no-spurious-pi} holds with $C_{\rm PI}=\sqrt{2}$ and $\alpha=1$. Both Theorems~\ref{thm: general-mdp-no-spurious} and~\ref{thm: general-mdp-PLK} apply.

    The pointwise condition in Theorem~\ref{thm: general-mdp-pointwise-PLK} nevertheless fails. Consider $\theta=(0,1)$, for which
    \begin{equation*}
        l(0,1)=l(\theta^*)=0, \qquad \nabla l(0,1)=(2,0).
    \end{equation*}
    Since $(-2,0)\in N_\Theta((0,1))$, the point $(0,1)$ is stationary for $l$ over $\Theta$ and $\mathcal R(0,1)=0$. However, at state $1$,
    \begin{equation*}
        J_{\pi_{(0,1)}}(1)-(TJ_{\pi_{(0,1)}})(1)=2-1=1.
    \end{equation*}
    Condition~\eqref{condition: general-mdp-pointwise-error} would therefore require $1\le C\mathcal R(0,1)^\alpha=0$, which is impossible for every finite $C$ and every $\alpha\in[1,2]$. The policy $\pi_{(0,1)}$ is optimal for $l$ because it remains in state $0$ when initialized from $\rho$, but it selects a suboptimal action at state $1$. Theorems~\ref{thm: general-mdp-no-spurious} and~\ref{thm: general-mdp-PLK} permit this distinction, whereas Condition~\eqref{condition: general-mdp-pointwise-error} requires Bellman optimality at every state whenever $\theta$ is a stationary point of $l$ over $\Theta$.
\end{example}

The next example shows that Theorem~\ref{thm: general-mdp-pointwise-PLK} can apply even when $\kappa_\rho=\infty$, in which case the concentrability assumption of Theorem~\ref{thm: general-mdp-PLK} fails.

\begin{figure}[htbp]
    \FIGURE
    {\includegraphics[width=0.35\textwidth]{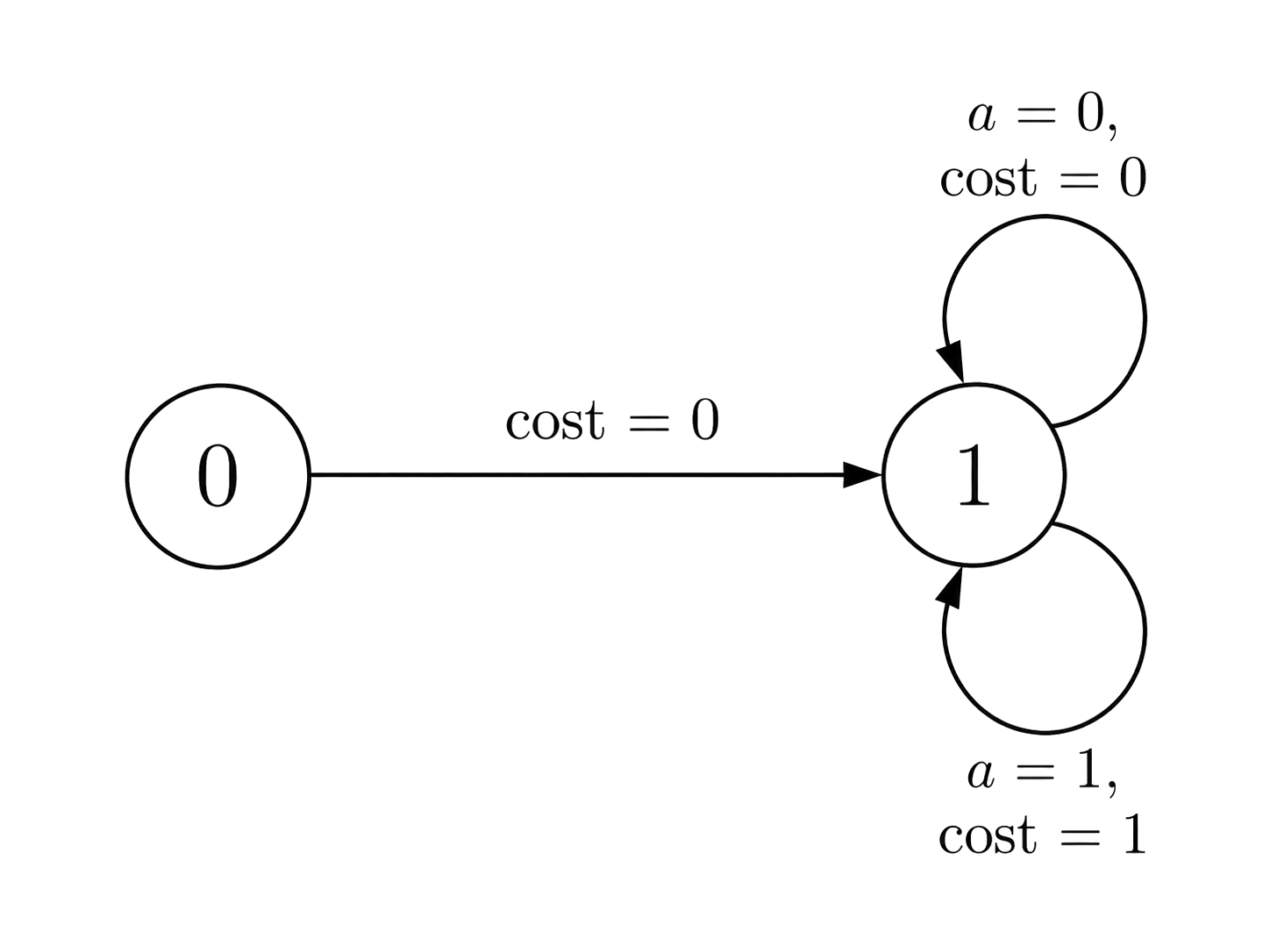}}
    {Two-State MDP with $\kappa_\rho=\infty$ in Example~\ref{example: without-concentrability}
    \label{fig: pointwise-without-concentrability}}
    {}
\end{figure}

\begin{example}\label{example: without-concentrability}
    Consider the two-state MDP in Figure~\ref{fig: pointwise-without-concentrability}, with $\gamma=1/2$ and $\rho(\{0\})=1$. At state $0$, the only action has zero cost and moves to state $1$. State $1$ is absorbing under both actions, with actions $0$ and $1$ incurring costs $0$ and $1$, respectively. As in Example~\ref{example: weighted-without-pointwise}, represent randomization between these actions by $a\in[0,1]$, the probability of selecting action $1$, so the expected one-period cost at state $1$ is $a$. Parameterize the policy at state $1$ by $\theta\in[0,1]$, with optimal parameter $\theta^*=0$. Then
    \begin{equation*}
        J_{\pi_\theta}(0)=\theta,\qquad J_{\pi_\theta}(1)=2\theta, \qquad l(\theta)=\theta/2.
    \end{equation*}
    The stationarity measure is $\mathcal R(0)=0$ and $\mathcal R(\theta)=1/2$ for $\theta>0$, while the Bellman residuals at states 0 and 1 are $0$ and $\theta$, respectively. Hence \eqref{condition: general-mdp-pointwise-error} holds with $\alpha=1$ and $C=2$. The finite state space, bounded costs, and affine value and policy-improvement expressions also verify Assumptions~\ref{assumption: general-mdp-technical} and~\ref{assumption: general-mdp-regularity}. However, the integral on the right-hand side of~\eqref{ineq: effective-concentrability} is zero for every $\theta$, while $l(\theta)-l(0)=\theta/2>0$ for $\theta>0$. Hence, $\kappa_\rho=\infty$.
\end{example}

Examples~\ref{example: weighted-without-pointwise} and~\ref{example: without-concentrability} show that Theorems~\ref{thm: general-mdp-PLK} and~\ref{thm: general-mdp-pointwise-PLK} provide complementary sufficient conditions for establishing the P{\L}K condition for policy optimization. Under the stated attainment assumption, Condition~\eqref{condition: general-mdp-pointwise-error} in Theorem~\ref{thm: general-mdp-pointwise-PLK} implies Conditions~\eqref{condition: general-mdp-approx-closure-pi} and~\eqref{condition: general-mdp-approx-no-spurious-pi} in Theorem~\ref{thm: general-mdp-PLK}, but the converse need not hold. However, Theorem~\ref{thm: general-mdp-pointwise-PLK} does not require the finite effective concentrability coefficient assumed in Theorem~\ref{thm: general-mdp-PLK}.

\section{Applications}\label{sec: application}
Our general results encompass settings covered by the exact-closure framework when the common standing assumptions in Section~\ref{sec: landscape} hold for the same policy class and parameter domain. We next verify the pointwise condition in Theorem~\ref{thm: general-mdp-pointwise-PLK} for inventory systems with Markov-modulated demand and stochastic cash-balance problems. For both models, the policy gradient objective satisfies the P{\L}K condition without requiring a finite effective concentrability coefficient or a bounded distribution-mismatch coefficient.

\subsection{Inventory Control with Markov-Modulated Demand}\label{sec: markov-demand}
We consider a Markov-modulated extension of the inventory model in Section~\ref{subsec: single-state-inventory}. The extension is motivated by applications in which demand may exhibit temporal dependence due to persistent factors such as economic conditions, market conditions, or seasonality. Following \citet{song1993inventory}, we capture such dependence by letting the demand distribution be modulated by an exogenous finite-state discrete-time Markov chain, whose state represents the underlying demand environment. Apart from this environment process, the ordering decision, inventory transition, and holding/backlogging cost retain the same structure as in Section~\ref{subsec: single-state-inventory}. The only modification is that the demand distribution, and hence the one-period expected cost, may depend on the current environment state.

Let $\ci$ denote the state space of the exogenous Markov chain. At the beginning of period $t$, the MDP state is $s_t=(x_t,i_t)\in(-\infty, B]\times\ci$, where $x_t$ is the pre-order inventory level, $i_t$ is the environment state, and $B>0$ is the warehouse capacity. As in Section~\ref{subsec: single-state-inventory}, the action is the post-order inventory level $y_t\in\ca_{(x_t,i_t)}=[x_t,B]$. Conditional on the current environment state $i_t=i$, the demand $D_t$ and the next environment state $i_{t+1}$ are independent of each other and of the history. The demand is nonnegative and follows the distribution $P_D(\cdot\mid i)$ with a finite mean, while the environment process evolves according to the transition matrix $p=(p_{ij})_{i,j\in\ci}$, where $p_{ij}=\mathbb P(i_{t+1}=j\mid i_t=i)$ and $\sum_{j\in\ci}p_{ij}=1$ for all $i\in\ci$. After $y_t$ is chosen and demand is realized, the next inventory level is $x_{t+1}=y_t-D_t$. Let $\rho$ and $\nu$ denote the initial distributions of $x_0$ and $i_0$, respectively. We assume that $x_0$ has a finite mean. By the independence of $x_0$ and $i_0$, the initial state distribution is $\rho\otimes\nu$.

The expected holding and backlogging cost incurred after ordering up to level $y$ in environment state $i$ is
\begin{equation*}
    C_i(y)\coloneqq \mathbb E_{D\sim P_D(\cdot\mid i)}\left[h(y-D)^+ + b(D-y)^+\right],
\end{equation*}
where $h\ge0$ and $b\ge0$ are the unit holding and backlogging costs, respectively. Thus, the one-period cost is $g((x,i),y)=C_i(y)$ for $y\in\ca_{(x,i)}$. Since the holding and backlogging cost is convex in $y$ for each demand realization, $C_i$ is convex for every $i\in\ci$. We assume zero ordering cost for simplicity. The same analysis applies to a unit ordering cost $c\ge0$ when $b\ge(1-\gamma)c$, so that ordering one unit costs no more than the discounted cost of backlogging it indefinitely. A value-function transformation replaces $(h,b)$ by $(h+(1-\gamma)c, b-(1-\gamma)c)$, up to a policy-independent term. If $b<(1-\gamma)c$, never ordering is optimal.

We focus on the state-dependent base-stock policy class, which is optimal for this model \citep{song1993inventory}. A parameter $\theta=(\theta_i)_{i\in\ci}$ specifies one base-stock level for each environment state:
\begin{equation*}
    \pi_\theta(x,i) = x\vee\theta_i, \qquad \Theta\coloneqq\left\{\theta\in\R^{|\ci|}:0\le\theta_i\le B,\ \forall i\in\ci\right\}, \qquad
    \Pi_\Theta\coloneqq\{\pi_\theta:\theta\in\Theta\}.
\end{equation*}
For $\theta\in\Theta$ and each environment state $i\in\ci$, define
\begin{equation*}
    G_{\theta,i}(y) \coloneqq C_i(y) + \gamma \sum_{j\in\ci}p_{ij} \mathbb E_{D\sim P_D(\cdot\mid i)} \left[ J_{\pi_\theta}(y-D,j) \right].
\end{equation*}
Then, $J_{\pi_\theta}(x,i)=G_{\theta,i}(x\vee\theta_i)$ by the Bellman equation. Moreover, for every $\theta'\in\Theta$,
\begin{equation*}
    \cb(\pi_{\theta'}\mid \eta_{\pi_\theta},J_{\pi_\theta}) = \int_{(-\infty,B]\times\ci} G_{\theta,i}(x\vee\theta'_i) \ \eta_{\pi_\theta}(dx,di).
\end{equation*}

We impose the following assumptions on the initial state and conditional demand distributions.

\begin{assumption}
    \label{assumption: inventory}
    Let $\rho(z)\coloneqq\mathbb P(x_0\le z)$ and $P_D(z\mid i)\coloneqq\mathbb P(D\le z\mid i)$. The following conditions hold.
    \begin{henumerate}
        \item \label{assumption: inventory-independence}
        The initial inventory level $x_0$ is independent of the initial environment state $i_0$ and the demand process.

        \item \label{assumption: inventory-lipschitz}
        The function $\rho$ is $L_\rho$-Lipschitz continuous, and $P_D(\cdot\mid i)$ is $L_D$-Lipschitz continuous for every $i\in\ci$. We denote their densities by $p_\rho$ and $p_D(\cdot\mid i)$, respectively.

        \item \label{assumption: exploratory-initial-distribution}
        There exist constants $\rho_{\min}>0$ and $\nu_{\min}>0$ such that $\rho(0)\ge\rho_{\min}$ and $\nu(i)\ge\nu_{\min}$ for every $i\in\ci$.

        \item \label{assumption: inventory-density-lower-bound}
        There exists $\mu_D>0$ such that $p_D(d\mid i)\ge\mu_D$ for every $d\in[0,B]$ and $i\in\ci$.
    \end{henumerate}
\end{assumption}

Assumption~\ref{assumption: inventory}.\ref{assumption: inventory-independence} is a standard independence condition in Markov-modulated demand models \citep{song1993inventory}. Assumption~\ref{assumption: inventory}.\ref{assumption: inventory-lipschitz} ensures that $l(\theta)$ is continuously differentiable; this condition is satisfied, for example, by many commonly used distributions with bounded densities, such as uniform, exponential, and Erlang distributions. Assumption~\ref{assumption: inventory}.\ref{assumption: exploratory-initial-distribution} ensures that each base-stock level affects the objective with uniformly positive probability, preventing its gradient from vanishing merely because of insufficient exploration. The condition holds, for example, when $x_0$ is uniformly distributed on $[-1,0]$ and $\nu$ is uniform over the environment states. Finally, Assumption~\ref{assumption: inventory}.\ref{assumption: inventory-density-lower-bound} guarantees that each one-period cost $C_i$ is strongly convex on $[0, B]$ when $h + b > 0$ and holds for distributions whose conditional densities are uniformly bounded away from zero on this interval, such as uniform and exponential distributions.

Under the first three parts of Assumption~\ref{assumption: inventory}, the model satisfies Assumptions~\ref{assumption: general-mdp-technical} and~\ref{assumption: general-mdp-regularity}, as established in Lemma~\ref{lemma: markov-demand-regularity} in Section~\ref{appendix: inventory} of the electronic companion. Classical Bellman-equation analysis of the Markov-modulated inventory model shows that, if $\theta^*$ parameterizes an optimal state-dependent base-stock policy, then $G_{\theta^*,i}$ is convex on $[0,B]$ for every $i\in\ci$ \citep{song1993inventory}. For an arbitrary parameter $\theta$, however, $J_{\pi_\theta}$ is the policy-evaluation value function associated with the fixed policy $\pi_\theta$, so the classical argument does not imply that $G_{\theta,i}$ is convex. The following lemma establishes a relaxation of this optimal-policy convexity for every $\theta\in\Theta$.

\begin{lemma}[Approximate Convexity] \label{lemma: markov-approx-convexity}
    Suppose the first three parts of Assumption~\ref{assumption: inventory} hold. Then, for every $\theta\in\Theta$, $i\in\ci$, and $0\le z_1<z_2\le B$,
    \begin{equation} \label{ineq: markov-G-slope-linear}
        G_{\theta,i}'(z_2)-G_{\theta,i}'(z_1) \ge -\frac{\gamma\mathcal R(\theta)} {\rho_{\min}\nu_{\min}(1-\gamma)^2}.
    \end{equation}
    In particular, if $\mathcal R(\theta)=0$, then $G_{\theta,i}$ is convex on $[0,B]$. If, in addition, Assumption~\ref{assumption: inventory}.\ref{assumption: inventory-density-lower-bound} holds, then
    \begin{equation} \label{ineq: markov-G-slope-quadratic}
        G_{\theta,i}'(z_2)-G_{\theta,i}'(z_1) \ge (h+b)\mu_D(z_2-z_1) - \frac{\gamma\mathcal R(\theta)} {\rho_{\min}\nu_{\min}(1-\gamma)^2}, \qquad 0\le z_1<z_2\le B.
    \end{equation}
\end{lemma}

Recall that a continuously differentiable function on an interval is convex if and only if its derivative is nondecreasing. Inequality~\eqref{ineq: markov-G-slope-linear} allows $G_{\theta,i}'$ to decrease, but bounds any such decrease by
\begin{equation*}
    \frac{\gamma\mathcal R(\theta)} {\rho_{\min}\nu_{\min}(1-\gamma)^2}.
\end{equation*}
This is the sense in which $G_{\theta, i}$ is approximately convex. Lemma~\ref{lemma: markov-approx-convexity} therefore extends the classical convexity conclusion beyond optimal base-stock parameters: $G_{\theta, i}$ is convex for every $i\in\ci$ whenever $\theta$ satisfies the first-order optimality condition for $l$ over $\Theta$. Under Assumption~\ref{assumption: inventory}.\ref{assumption: inventory-density-lower-bound}, inequality~\eqref{ineq: markov-G-slope-quadratic} provides the corresponding approximate strong-convexity property with curvature $(h+b)\mu_D$.

The key mechanism behind Lemma~\ref{lemma: markov-approx-convexity} is the propagation of approximate convexity through the Bellman equation. To see this, let $\Delta_J$ denote the largest possible decrease in the slope of $J_{\pi_\theta}(\cdot, i)$ across all inventory and environment states, and define $\Delta_G$ analogously for $G_{\theta, i}$. The convex one-period cost does not create additional decrease in slope. The expectation over random demand and the average over the next environment state do not amplify a decrease inherited from the future value function. The discount factor further reduces its magnitude, so $\Delta_G\leq\gamma\Delta_J$. Under the base-stock policy, the slope of $J_{\pi_\theta}(\cdot,i) = G_{\theta,i}(\cdot \vee \theta_i)$ is zero below $\theta_i$ and equals $G_{\theta,i}'$ above $\theta_i$. If $\theta_i<B$, crossing $\theta_i$ changes the slope from zero to $G_{\theta,i}'(\theta_i)$. A negative value of $G_{\theta,i}'(\theta_i)$ therefore creates a decrease in slope, whose magnitude is bounded by
\begin{equation*}
    \frac{\mathcal R(\theta)}{\rho_{\min}\nu_{\min}(1-\gamma)}.
\end{equation*}
Consequently,
\begin{equation*}
    \Delta_J \le \Delta_G + \frac{\mathcal R(\theta)} {\rho_{\min}\nu_{\min}(1-\gamma)} \le \gamma\Delta_J + \frac{\mathcal R(\theta)} {\rho_{\min}\nu_{\min}(1-\gamma)}.
\end{equation*}
Solving this inequality and using $\Delta_G\le\gamma\Delta_J$ yields
\begin{equation*}
    \Delta_G \le \frac{\gamma\mathcal R(\theta)} {\rho_{\min}\nu_{\min}(1-\gamma)^2},
\end{equation*}
which explains \eqref{ineq: markov-G-slope-linear}. Intuitively, the base-stock thresholds may introduce downward slope jumps, while errors inherited from future periods are repeatedly discounted. When $\mathcal R(\theta)=0$, the inequalities imply $\Delta_J=\Delta_G=0$, and exact convexity is recovered. We defer the proof to Appendix~\ref{appendix: inventory-core}.

Using Lemma~\ref{lemma: markov-approx-convexity}, we verify Condition~\eqref{condition: general-mdp-pointwise-error} and apply Theorem~\ref{thm: general-mdp-pointwise-PLK} to obtain the following result.

\begin{theorem}[P\L{}K Condition]\label{thm: markov-demand-PLK}
    Suppose the first three parts of Assumption~\ref{assumption: inventory} hold. Let $\theta^*\in\Theta$ be an optimal base-stock parameter. Then the following statements hold.
    \begin{enumerate}[label=(\roman*)]
        \item\label{markov-demand-no-suboptimal-stationary} For every $\theta\in\Theta$,
        \begin{equation}\label{ineq: markov-PLK-linear}
            l(\theta)-l(\theta^*)\le\frac{B}{\rho_{\min}\nu_{\min}(1-\gamma)^2}\mathcal R(\theta).
        \end{equation}
        Thus, $l$ satisfies the P{\L}K condition with exponent $1$ on $\Theta$.

        \item If, in addition, Assumption~\ref{assumption: inventory}.\ref{assumption: inventory-density-lower-bound} holds and $h+b>0$, then, for every $\theta\in\Theta$,
        \begin{equation}\label{ineq: markov-PLK}
            l(\theta)-l(\theta^*)\le\frac{1}{2(h+b)\mu_D\rho_{\min}^2\nu_{\min}^2(1-\gamma)^4}\mathcal R(\theta)^2.
        \end{equation}
        Thus, $l$ satisfies the P{\L}K condition with exponent $2$ on $\Theta$.
    \end{enumerate}
    In particular, part~(i) implies that every stationary point of $l$ over $\Theta$ is globally optimal over $\Theta$.
\end{theorem}

Lemma~\ref{lemma: markov-demand-smoothness} in Section~\ref{appendix: inventory} of the electronic companion establishes that $\nabla l$ is Lipschitz continuous on $\Theta$ under the first three parts of Assumption~\ref{assumption: inventory}. Let $L>0$ be a Lipschitz constant for $\nabla l$ on $\Theta$. Combining Theorem~\ref{thm: markov-demand-PLK} with Lemma~\ref{lemma: convergence rate}, projected gradient descent using exact policy gradients and step size $1/L$ attains an $\epsilon$-optimal policy in $\co(1/\epsilon)$ iterations under the assumptions of part~(i). Under the additional density lower bound and $h+b>0$, part~(ii) yields linear convergence and an iteration complexity of $\co(\log(1/\epsilon))$.

\subsection{Stochastic Cash-Balance Problem}\label{sec: cash-balance}
The stochastic cash-balance problem was originally formulated to characterize a firm's optimal cash-holding decisions for meeting transaction requirements and can also be interpreted as an inventory control problem involving rented equipment \citep{whisler1967stochastic,chen2009new}. Under the inventory interpretation, the cash-balance problem constitutes a two-sided extension of the inventory model studied in Section~\ref{subsec: single-state-inventory}. Unlike the standard inventory model, which permits only nonnegative order quantities, the cash-balance problem allows both upward and downward adjustments of the inventory level. Except for this two-sided adjustment mechanism and the associated transaction costs, the state-transition dynamics and the holding and backlogging cost structure are identical to those of the inventory model.

Given the current inventory level $x\in\R$, the decision maker selects a post-adjustment inventory level $y\in[\underline B,\bar B]$. Moving the inventory level upward or downward incurs the transaction cost
\begin{equation*}
    c(y,x)\coloneqq
    \begin{cases}
        k(y-x), & y\ge x,\\
        q(x-y), & y<x,
    \end{cases}
\end{equation*}
where $k+q\ge0$. The demands across periods are independent and identically distributed copies of a real-valued random variable $D$ with a finite mean. After the adjustment, demand $D$ is realized and the next state is $y-D$. The expected holding and backlogging cost is
\begin{equation*}
    C(y)\coloneqq \E\left[h(y-D)^+ + p(D-y)^+\right],
\end{equation*}
where $h,p\ge0$. Thus, the one-period cost is $g(x,y)=c(y,x)+C(y)$. Let $\rho$ denote the distribution of the initial state $x_0$, and assume that $\E[|x_0|]<\infty$. We use the value function, policy gradient objective, discounted state-occupancy measure, and state-action value function defined in Section~\ref{sec: problem formulation}.

Because the cash level can be adjusted in either direction, the one-sided base-stock policy in the inventory model is replaced by the two-sided base-stock policy
\begin{equation*}
    \pi_\theta(x)=(x\vee\underline\theta)\wedge\bar\theta, \qquad \Theta\coloneqq \left\{ (\underline\theta,\bar\theta): \underline B\le\underline\theta\le\bar\theta\le\bar B \right\}, \qquad \Pi_\Theta\coloneqq\{\pi_\theta:\theta\in\Theta\}.
\end{equation*}
This policy raises the cash level to $\underline\theta$ when it is too low, leaves it unchanged when it lies between the two thresholds, and lowers it to $\bar\theta$ when it is too high. The two-sided base-stock policy class contains an optimal policy for this model \citep{whisler1967stochastic,eppen1969cash}; throughout this subsection, the bounds $\underline B$ and $\bar B$ are chosen so that an optimal two-sided base-stock parameter belongs to $\Theta$.

For every $\theta\in\Theta$, define
\begin{equation}\label{eq: cash-G}
    G_\theta(y)\coloneqq C(y)+\gamma\E\left[J_{\pi_\theta}(y-D)\right], \qquad y\in[\underline B,\bar B].
\end{equation}
Then
\begin{equation}\label{eq: cash-Q-G}
    Q_{\pi_\theta}(x,y) = c(y,x)+G_\theta(y), \qquad x\in\R,\quad y\in[\underline B,\bar B],
\end{equation}
and the Bellman equation gives
\begin{equation}\label{eq: cash-bellman-J-G}
    J_{\pi_\theta}(x) = c(\pi_\theta(x),x)+G_\theta(\pi_\theta(x)).
\end{equation}
Consequently, for every $\theta'\in\Theta$,
\begin{equation}\label{eq: cash-weighted-pi-G}
    \cb(\pi_{\theta'}\mid\eta_{\pi_\theta},J_{\pi_\theta}) = \int_\R \left[ c(\pi_{\theta'}(x),x) + G_\theta(\pi_{\theta'}(x)) \right] \eta_{\pi_\theta}(dx).
\end{equation}

To establish the landscape properties of $l$, we impose the following assumptions, which are the two-sided counterparts of those used for the inventory system.

\begin{assumption}\label{assumption: cash balance}
    The following conditions hold.
    \begin{henumerate}
        \item \label{assumption: cash balance-independence}
        The initial state $x_0$ is independent of the demand process.

        \item \label{assumption: cash balance-Lipschitz}
        The cumulative distribution functions of $x_0$ and $D$ are $L_\rho$-Lipschitz continuous and $L_D$-Lipschitz continuous, respectively.

        \item \label{assumption: cash balance-exploratory-initial-distribution}
        There exists $\alpha>0$ such that
        \begin{equation*}
            \rho((-\infty,\underline B])\ge\alpha, \qquad \rho([\bar B,\infty))\ge\alpha.
        \end{equation*}

        \item \label{assumption: cash balance-density-lower-bound}
        The demand admits a density that is uniformly bounded below by $\mu_D>0$ on $[\underline B,\bar B]$.
    \end{henumerate}
\end{assumption}

Assumption~\ref{assumption: cash balance}.\ref{assumption: cash balance-independence} is the standard independence condition for the cash-balance model. Assumption~\ref{assumption: cash balance}.\ref{assumption: cash balance-Lipschitz} guarantees the differentiability needed for the Policy Gradient Theorem. Assumption~\ref{assumption: cash balance}.\ref{assumption: cash balance-exploratory-initial-distribution} ensures that both the lower and upper thresholds affect the objective with uniformly positive probability, preventing their gradients from vanishing merely because of insufficient exploration. The uniform probability lower bound $\alpha$ is used to establish the P{\L}K bounds. For example, the condition holds with $\alpha=1/2$ when the initial cash balance follows an equal mixture of uniform distributions on $[\underline B-1,\underline B]$ and $[\bar B,\bar B+1]$. Finally, Assumption~\ref{assumption: cash balance}.\ref{assumption: cash balance-density-lower-bound} provides the curvature needed for the quadratic P\L{}K bound. Under the first three parts of Assumption~\ref{assumption: cash balance}, Lemma~\ref{lemma: cash-balance-regularity} in Section~\ref{appendix: cash-balance} of the electronic companion verifies Assumptions~\ref{assumption: general-mdp-technical} and~\ref{assumption: general-mdp-regularity}. We next establish approximate convexity of $G_\theta$.

\begin{lemma}[Approximate Convexity]\label{lemma: cash-approx-convexity}
    Suppose the first three parts of Assumption~\ref{assumption: cash balance} hold. Then, for every $\theta\in\Theta$ and $\underline B\le z_1<z_2\le\bar B$,
    \begin{equation}\label{ineq: cash-G-slope-linear}
        G_\theta'(z_2)-G_\theta'(z_1) \ge -\frac{\gamma\mathcal R(\theta)}{\alpha(1-\gamma)^2}.
    \end{equation}
    In particular, if $\mathcal R(\theta)=0$, then $G_\theta$ is convex on $[\underline B,\bar B]$. If, in addition, Assumption~\ref{assumption: cash balance}.\ref{assumption: cash balance-density-lower-bound} holds, then
    \begin{equation}\label{ineq: cash-G-slope-quadratic}
        G_\theta'(z_2)-G_\theta'(z_1) \ge (h+p)\mu_D(z_2-z_1) -\frac{\gamma\mathcal R(\theta)}{\alpha(1-\gamma)^2}, \qquad \underline B\le z_1<z_2\le\bar B.
    \end{equation}
\end{lemma}

Using Lemma~\ref{lemma: cash-approx-convexity}, we verify Condition~\eqref{condition: general-mdp-pointwise-error} and apply Theorem~\ref{thm: general-mdp-pointwise-PLK} to obtain the following result.

\begin{theorem}[P\L{}K Condition]\label{thm: cash-balance-PLK}
    Suppose the first three parts of Assumption~\ref{assumption: cash balance} hold. Let $\theta^*\in\Theta$ be an optimal two-sided base-stock parameter. Then the following statements hold.
    \begin{enumerate}[label=(\roman*)]
        \item For every $\theta\in\Theta$,
        \begin{equation}\label{ineq: cash-balance-PLK-linear}
            l(\theta)-l(\theta^*)\le\frac{\bar B-\underline B}{\alpha(1-\gamma)^2}\mathcal R(\theta).
        \end{equation}
        Thus, $l$ satisfies the P{\L}K condition with exponent $1$ on $\Theta$.

        \item If, in addition, Assumption~\ref{assumption: cash balance}.\ref{assumption: cash balance-density-lower-bound} holds and $h+p>0$, then, for every $\theta\in\Theta$,
        \begin{equation}\label{ineq: cash-balance-PLK}
            l(\theta)-l(\theta^*)\le\frac{1}{2(h+p)\mu_D\alpha^2(1-\gamma)^4}\mathcal R(\theta)^2.
        \end{equation}
        Thus, $l$ satisfies the P{\L}K condition with exponent $2$ on $\Theta$.
    \end{enumerate}
    In particular, part~(i) implies that every stationary point of $l$ over $\Theta$ is globally optimal over $\Theta$.
\end{theorem}

Lemma~\ref{lemma: cash-balance-smoothness} in Section~\ref{appendix: cash-balance} of the electronic companion establishes that $\nabla l$ is Lipschitz continuous on $\Theta$ under the first three parts of Assumption~\ref{assumption: cash balance}. Let $L>0$ be a Lipschitz constant for $\nabla l$ on $\Theta$. Combining Theorem~\ref{thm: cash-balance-PLK} with Lemma~\ref{lemma: convergence rate}, projected gradient descent using exact policy gradients and step size $1/L$ attains an $\epsilon$-optimal policy in $\co(1/\epsilon)$ iterations under the assumptions of part~(i). Under the additional density lower bound and $h+p>0$, part~(ii) yields linear convergence and an iteration complexity of $\co(\log(1/\epsilon))$.

\section{Conclusion}\label{sec: conclusion} 
We study the optimization landscape for infinite-horizon discounted MDPs with general states and actions under structured policy classes. An inventory example shows that closure under policy improvement can fail even when the policy class contains an optimal policy. We establish weaker conditions that require exact weighted policy-improvement properties only at stationary points of the policy gradient objective and use stationarity-controlled approximations of these properties to derive P{\L}K conditions for this objective when the effective concentrability coefficient is finite. A separate pointwise condition yields P{\L}K bounds without a concentrability assumption. Under the common standing assumptions and for the same policy class and parameter domain, our results encompass the corresponding exact-closure results. We verify the pointwise condition for inventory systems with Markov-modulated demand and stochastic cash-balance problems. For both models, we establish that the policy gradient objective satisfies the P{\L}K condition with exponent $1$ and, under additional curvature conditions, exponent $2$. Together with the smoothness proved under the baseline assumptions, these results give an $\co(1/\epsilon)$ iteration complexity and linear convergence, respectively, for projected gradient descent using exact policy gradients.

Our work opens several directions for future research. First, it remains open to establish the optimal sample complexity for stochastic policy gradient methods in structured operations models. Existing literature combines P{\L}K or related landscape conditions with variance reduction techniques to derive optimal rates \citep{fatkhullin2022sharp,masiha2026optimal}. Applying these variance-reduction methods to base-stock policies requires further analysis. Second, exploiting additional model structure may sharpen how the application bounds depend on the discount factor and the exploration constants. Third, it remains open whether the same landscape results extend to MDPs with nonsmooth objectives, e.g., inventory systems with discrete demands. Finally, extending our results to long-run average, constrained, and partially observed MDPs may broaden the applicability to other structured operations and control problems.

\renewcommand{\theHsection}{A\arabic{section}}
\begin{APPENDICES}
\normalsize

\section{Proofs of the General Landscape Results}\label{appendix: landscape}

Under Assumption~\ref{assumption: general-mdp-technical}, the Performance Difference Lemma \citep{kakade2002approximately}
\begin{equation} \label{eq: general-mdp-performance-difference}
    l(\theta)-l(\theta') = \int_{\cs} \left( J_{\pi_\theta}(s) - T_{\pi_{\theta'}}J_{\pi_\theta}(s) \right) \eta_{\pi_{\theta'}}(ds)
\end{equation}
holds for every $\theta,\theta'\in\Theta$. Its standard proof remains valid since Assumption~\ref{assumption: general-mdp-technical} guarantees the Bellman equations and the absolute integrability of all terms involved.

\begin{proof}{Proof of Theorem~\ref{thm: general-mdp-no-spurious}}
    Fix a stationary point $\theta$ of $l$ over $\Theta$. By \eqref{eq: general-mdp-policy-gradient},
    \begin{equation*}
        \nabla l(\theta) = \nabla_{\theta'} \cb(\pi_{\theta'} \mid \eta_{\pi_\theta},J_{\pi_\theta}) \bigg|_{\theta'=\theta}.
    \end{equation*}
    Since $\theta$ is stationary for $l$ over $\Theta$, it is also a stationary point of $\theta'\mapsto \cb(\pi_{\theta'} \mid \eta_{\pi_\theta},J_{\pi_\theta})$ over $\Theta$.
    By condition~\ref{condition: pointwise-no-spurious-pi},
    \begin{equation*}
        \cb(\pi_\theta \mid \eta_{\pi_\theta},J_{\pi_\theta}) = \min_{\theta'\in\Theta} \cb(\pi_{\theta'} \mid \eta_{\pi_\theta},J_{\pi_\theta}).
    \end{equation*}
    Combining this equality with condition~\ref{condition: pointwise-closure-pi} gives
    \begin{equation*}
        \cb(\pi_\theta \mid \eta_{\pi_\theta},J_{\pi_\theta}) = \min_{\theta'\in\Theta} \cb(\pi_{\theta'} \mid \eta_{\pi_\theta},J_{\pi_\theta}) = \min_{\pi'\in\Pi} \cb(\pi' \mid \eta_{\pi_\theta},J_{\pi_\theta}) = \int_\cs (TJ_{\pi_\theta})(s)\eta_{\pi_\theta}(ds).
    \end{equation*}
    On the other hand, since $J_{\pi_\theta}=T_{\pi_\theta}J_{\pi_\theta}$,
    \begin{equation*}
        \cb(\pi_\theta \mid \eta_{\pi_\theta},J_{\pi_\theta}) = \int_\cs (T_{\pi_\theta}J_{\pi_\theta})(s)\eta_{\pi_\theta}(ds) = \int_\cs J_{\pi_\theta}(s)\eta_{\pi_\theta}(ds).
    \end{equation*}
    Hence
    \begin{equation*}
        \int_\cs \left(J_{\pi_\theta}-TJ_{\pi_\theta}\right)(s) \eta_{\pi_\theta}(ds) = 0.
    \end{equation*}
    Since $J_{\pi_\theta}\succeq TJ_{\pi_\theta}$ and $\eta_{\pi_\theta}\succeq (1-\gamma)\rho$, we obtain
    \begin{equation*}
        0 = \int_\cs \left(J_{\pi_\theta}-TJ_{\pi_\theta}\right)d\eta_{\pi_\theta} \ge (1-\gamma) \int_\cs \left(J_{\pi_\theta}-TJ_{\pi_\theta}\right)d\rho \ge 0.
    \end{equation*}
    Therefore,
    \begin{equation*}
        \int_\cs \left(J_{\pi_\theta}-TJ_{\pi_\theta}\right)d\rho = 0.
    \end{equation*}
    Since $J_{\pi_\theta}\succeq TJ_{\pi_\theta}$, it follows that $J_{\pi_\theta}=TJ_{\pi_\theta}$ ($\rho$-almost surely). Absolute continuity of $\eta_{\pi_{\theta^*}}$ with respect to $\rho$ then implies that $J_{\pi_\theta}=TJ_{\pi_\theta}$ ($\eta_{\pi_{\theta^*}}$-almost surely).

    By the Performance Difference Lemma \eqref{eq: general-mdp-performance-difference},
    \begin{equation*}
        l(\theta)-l(\theta^*) = \int_\cs \left( J_{\pi_\theta}(s)-T_{\pi_{\theta^*}}J_{\pi_\theta}(s) \right) \eta_{\pi_{\theta^*}}(ds).
    \end{equation*}
    Since $T_{\pi_{\theta^*}}J_{\pi_\theta}\succeq TJ_{\pi_\theta}$, we have
    \begin{equation*}
        0 \le l(\theta)-l(\theta^*) \le \int_\cs \left(J_{\pi_\theta}- TJ_{\pi_\theta}\right)(s) \eta_{\pi_{\theta^*}}(ds) = 0.
    \end{equation*}
    Thus $l(\theta)=l(\theta^*)$. Since $\theta^*$ is globally optimal over $\Theta$, the proof is complete. \Halmos
\end{proof}

\begin{proof}{Proof of Theorem~\ref{thm: general-mdp-PLK}}
    Fix $\theta\in\Theta$. The minimum of $\cb(\pi'\mid\eta_{\pi_\theta},J_{\pi_\theta})$ over $\Pi$ and $\theta' \mapsto \cb(\pi_{\theta'}\mid\eta_{\pi_\theta},J_{\pi_\theta})$ over $\Theta$ are finite. By the existence of a measurable policy attaining the minimum in $TJ_{\pi_\theta}$ at every state,
    \begin{equation*}
        \begin{aligned}
            \int_\cs\left[J_{\pi_\theta}(s)-(TJ_{\pi_\theta})(s)\right]\eta_{\pi_\theta}(ds)
            &=\cb(\pi_\theta\mid\eta_{\pi_\theta},J_{\pi_\theta})-\min_{\pi'\in\Pi}\cb(\pi'\mid\eta_{\pi_\theta},J_{\pi_\theta})\\
            &=\left[\cb(\pi_\theta\mid\eta_{\pi_\theta},J_{\pi_\theta})-\min_{\theta'\in\Theta}\cb(\pi_{\theta'}\mid\eta_{\pi_\theta},J_{\pi_\theta})\right]\\
            &\qquad+\left[\min_{\theta'\in\Theta}\cb(\pi_{\theta'}\mid\eta_{\pi_\theta},J_{\pi_\theta})-\min_{\pi'\in\Pi}\cb(\pi'\mid\eta_{\pi_\theta},J_{\pi_\theta})\right]\\
            &\le(C_{\rm PI}+C_{\rm cl})\mathcal R(\theta)^\alpha,
        \end{aligned}
    \end{equation*}
    where the inequality follows from Conditions~\eqref{condition: general-mdp-approx-closure-pi} and~\eqref{condition: general-mdp-approx-no-spurious-pi}.

    By the definition of $l$,
    \begin{equation*}
        l(\theta)-l(\theta^*) =(1-\gamma)\int_\cs\left[J_{\pi_\theta}(s)-J_{\pi_{\theta^*}}(s)\right]\rho(ds).
    \end{equation*}
    Multiplying the inequality in Definition~\ref{def: effective-concentrability} by $1-\gamma$ therefore gives
    \begin{equation*}
        l(\theta)-l(\theta^*) \le\kappa_\rho\int_\cs\left[J_{\pi_\theta}(s)-(TJ_{\pi_\theta})(s)\right]\rho(ds).
    \end{equation*}
    The Bellman equation gives $J_{\pi_\theta}=T_{\pi_\theta}J_{\pi_\theta}\ge TJ_{\pi_\theta}$, so the integrand is nonnegative. Moreover, the $t=0$ term in the discounted occupancy measure implies $\eta_{\pi_\theta}\succeq(1-\gamma)\rho$. Consequently,
    \begin{equation*}
        \int_\cs\left[J_{\pi_\theta}(s)-(TJ_{\pi_\theta})(s)\right]\rho(ds) \le\frac{1}{1-\gamma}\int_\cs\left[J_{\pi_\theta}(s)-(TJ_{\pi_\theta})(s)\right]\eta_{\pi_\theta}(ds).
    \end{equation*}
    Combining these inequalities with the previously established bound on the integral under $\eta_{\pi_\theta}$ yields
    \begin{equation*}
        l(\theta)-l(\theta^*) \le\frac{\kappa_\rho(C_{\rm PI}+C_{\rm cl})}{1-\gamma}\mathcal R(\theta)^\alpha.
    \end{equation*}
    This completes the proof.\Halmos
\end{proof}

\begin{proof}{Proof of Theorem~\ref{thm: general-mdp-pointwise-PLK}}
    Fix $\theta\in\Theta$. By the Performance Difference Lemma~\eqref{eq: general-mdp-performance-difference} and the definition of the Bellman optimality operator,
    \begin{equation*}
        \begin{aligned}
            l(\theta)-l(\theta^*) =\int_\cs\left[J_{\pi_\theta}(s)-(T_{\pi_{\theta^*}}J_{\pi_\theta})(s)\right]\eta_{\pi_{\theta^*}}(ds) \le\int_\cs\left[J_{\pi_\theta}(s)-(TJ_{\pi_\theta})(s)\right]\eta_{\pi_{\theta^*}}(ds) \le C\mathcal R(\theta)^\alpha,
        \end{aligned}
    \end{equation*}
    where the last inequality follows from~\eqref{condition: general-mdp-pointwise-error} and the fact that $\eta_{\pi_{\theta^*}}$ is a probability measure. \Halmos
\end{proof}

\section{Core Inventory Analysis}\label{appendix: inventory-core}

\begin{proof}{Proof of Lemma~\ref{lemma: markov-approx-convexity}}
    Fix $\theta\in\Theta$. By Lemma~\ref{lemma: markov-demand-regularity}, $J_{\pi_\theta}(\cdot,i)$ is Lipschitz continuous on $(-\infty,B]$ and $G_{\theta,i}$ is continuously differentiable on $[0,B]$ for every $i\in\ci$. For every $i\in\ci$ and $u\in[0,B]$,
    \begin{equation*}
        \eta_{\pi_\theta}\left((-\infty,u)\times\{i\}\right) \ge (1-\gamma)\mathbb P(x_0<u,\ i_0=i) = (1-\gamma)\mathbb P(x_0<u)\nu(i),
    \end{equation*}
    where the equality follows from Assumption~\ref{assumption: inventory}.\ref{assumption: inventory-independence}. By Assumption~\ref{assumption: inventory}.\ref{assumption: exploratory-initial-distribution}, we have
    \begin{equation*}
        \mathbb P(x_0<u) \ge \mathbb P(x_0<0) = \rho(0) \ge \rho_{\min}
    \end{equation*}
    and $\nu(i)\ge\nu_{\min}$. Consequently,
    \begin{equation} \label{eq: markov-uniform-state-visitation}
        \eta_{\pi_\theta}\left((-\infty,u)\times\{i\}\right) \ge (1-\gamma)\rho_{\min}\nu_{\min}, \qquad u\in[0,B],\ i\in\ci.
    \end{equation}

    For every $i\in\ci$, $u\in[0,B]$, and $g\in N_\Theta(\theta)$, the normal-cone property implies that $g_i(u-\theta_i)\leq 0$. Hence, by \eqref{eq: general-mdp-policy-gradient},
    \begin{equation*}
        \begin{aligned}
            \eta_{\pi_\theta}\left((-\infty,\theta_i)\times\{i\}\right) G_{\theta,i}'(\theta_i)(u-\theta_i) &= \partial_{\theta_i}l(\theta)(u-\theta_i) \\
            &= \left(\partial_{\theta_i}l(\theta)+g_i\right)(u-\theta_i) -g_i(u-\theta_i) \\
            & \ge -\left\|\nabla l(\theta)+g\right\|_2|u-\theta_i|.
        \end{aligned}
    \end{equation*}
    Choosing $g\in N_\Theta(\theta)$ that minimizes $\|\nabla l(\theta)+g\|_2$ yields
    \begin{equation*}
        \eta_{\pi_\theta}\left((-\infty,\theta_i)\times\{i\}\right) G_{\theta,i}'(\theta_i)(u-\theta_i) \ge -\mathcal R(\theta)|u-\theta_i|.
    \end{equation*}
    Combining this inequality with \eqref{eq: markov-uniform-state-visitation} yields 
    \begin{equation} \label{ineq: markov-residual-derivative-bounds} 
        G_{\theta,i}'(\theta_i)(u-\theta_i) \ge -\frac{\mathcal R(\theta)} {\rho_{\min}\nu_{\min}(1-\gamma)} |u-\theta_i|, \qquad u\in[0,B],\ i\in\ci. 
    \end{equation}

    We next control the nonconvexity of $J_{\pi_\theta}$. Let $\partial_+$ and $\partial_-$ denote the right and left derivatives, respectively, and define
    \begin{equation*}
        m\coloneqq \inf_{i\in\ci}\inf_{x<y\le B} \left\{ \partial_-J_{\pi_\theta}(y,i) - \partial_+J_{\pi_\theta}(x,i) \right\}.
    \end{equation*}
    The Lipschitz continuity of $J_{\pi_\theta}(\cdot,i)$ implies that $m>-\infty$. Moreover, $m\le0$ because $J_{\pi_\theta}(x,i) = G_{\theta,i}(\theta_i)$ for $x \le \theta_i$, so $J_{\pi_\theta}(\cdot,i)$ is constant on $(-\infty,\theta_i]$.

    Since $J_{\pi_\theta}(x,j)=G_{\theta,j}(x\vee\theta_j)$, the function $J_{\pi_\theta}(\cdot,j)$ is differentiable except at $\theta_j$. Because $P_D(\cdot\mid i)$ admits a density, the event $z-D=\theta_j$ has zero measure for every $z\in[0,B]$. The boundedness of the one-sided derivatives therefore allows differentiation under the expectation. Thus, for every $i\in\ci$ and $0\le z_1<z_2\le B$,
    \begin{equation} \label{ineq: markov-G-slope-bootstrap}
        G_{\theta,i}'(z_2)-G_{\theta,i}'(z_1) = C_i'(z_2)-C_i'(z_1) + \gamma\sum_{j\in\ci}p_{ij} \mathbb E_{D\sim P_D(\cdot\mid i)} \left[ \partial_-J_{\pi_\theta}(z_2-D,j) - \partial_+J_{\pi_\theta}(z_1-D,j) \right] \ge \gamma m,
    \end{equation}
    where the inequality follows from the convexity of $C_i$ and the definition of $m$.

    Fix $i\in\ci$ and $x<y\le B$. Since $J_{\pi_\theta}(z,i) = G_{\theta,i}(z\vee\theta_i)$, we have
    \begin{equation*}
        \partial_-J_{\pi_\theta}(y,i) - \partial_+J_{\pi_\theta}(x,i) = G_{\theta,i}'(y\vee\theta_i) - G_{\theta,i}'(x\vee\theta_i) + \mathbf 1_{\{x<\theta_i<y\}} G_{\theta,i}'(\theta_i).
    \end{equation*}
    If $x\vee\theta_i<y\vee\theta_i$, the first difference on the right-hand side is bounded below by $\gamma m$ by \eqref{ineq: markov-G-slope-bootstrap}. If the two arguments coincide, the difference is zero, which is also bounded below by $\gamma m$ because $m\le0$. If $x<\theta_i<y$, then \eqref{ineq: markov-residual-derivative-bounds} with $u=y$ gives
    \begin{equation*}
        G_{\theta,i}'(\theta_i) \ge -\frac{\mathcal R(\theta)} {\rho_{\min}\nu_{\min}(1-\gamma)}.
    \end{equation*}
    Therefore,
    \begin{equation*}
        \partial_-J_{\pi_\theta}(y,i) - \partial_+J_{\pi_\theta}(x,i) \ge \gamma m - \frac{\mathcal R(\theta)} {\rho_{\min}\nu_{\min}(1-\gamma)}.
    \end{equation*}
    Taking the infimum over $i$, $x$, and $y$ gives
    \begin{equation*}
        m \ge \gamma m - \frac{\mathcal R(\theta)} {\rho_{\min}\nu_{\min}(1-\gamma)},
    \end{equation*}
    and hence
    \begin{equation} \label{ineq: markov-approx-convexity}
        \partial_-J_{\pi_\theta}(y,i) - \partial_+J_{\pi_\theta}(x,i) \ge -\frac{\mathcal R(\theta)} {\rho_{\min}\nu_{\min}(1-\gamma)^2}, \qquad i\in\ci,\quad x<y\le B.
    \end{equation}

    Using \eqref{ineq: markov-approx-convexity} in the derivative representation of $G_{\theta,i}$ and using the convexity of $C_i$, we obtain
    \begin{equation*}
        G_{\theta,i}'(z_2)-G_{\theta,i}'(z_1) \ge -\frac{\gamma\mathcal R(\theta)} {\rho_{\min}\nu_{\min}(1-\gamma)^2}, \qquad 0\le z_1<z_2\le B.
    \end{equation*}
    In particular, if $\mathcal R(\theta)=0$, then $G_{\theta,i}'$ is nondecreasing on $[0,B]$, so $G_{\theta,i}$ is convex on $[0,B]$ for every $i\in\ci$.

    If Assumption~\ref{assumption: inventory}.\ref{assumption: inventory-density-lower-bound} holds, then
    \begin{equation*}
        C_i'(z_2)-C_i'(z_1) = (h+b)\left(P_D(z_2\mid i)-P_D(z_1\mid i)\right) \ge (h+b)\mu_D(z_2-z_1).
    \end{equation*}
    Together with \eqref{ineq: markov-approx-convexity}, this gives
    \begin{equation*}
        G_{\theta,i}'(z_2)-G_{\theta,i}'(z_1) \ge (h+b)\mu_D(z_2-z_1) - \frac{\gamma\mathcal R(\theta)} {\rho_{\min}\nu_{\min}(1-\gamma)^2}, \qquad 0\le z_1<z_2\le B.
    \end{equation*}
    This completes the proof. \Halmos
\end{proof}

\begin{proof}{Proof of Theorem~\ref{thm: markov-demand-PLK}}
    Fix $\theta\in\Theta$. By Lemma~\ref{lemma: markov-demand-regularity}, Assumptions~\ref{assumption: general-mdp-technical} and~\ref{assumption: general-mdp-regularity} hold.

    We first prove part~(i). Fix $i\in\ci$, $x\le B$, and a feasible action $y\in[x,B]$. If $y<0$, then $x\le y<0$, so $0$ is also feasible. Since demand is nonnegative and $\theta_j\ge0$ for every $j\in\ci$,
    \begin{equation*}
        G_{\theta,i}(y)-G_{\theta,i}(0)=-by\ge0.
    \end{equation*}
    Thus, any upper bound on $G_{\theta,i}(x\vee\theta_i)-G_{\theta,i}(0)$ also bounds $G_{\theta,i}(x\vee\theta_i)-G_{\theta,i}(y)$, and it suffices to consider $y\in[x,B]\cap[0,B]$.

    If $y>x\vee\theta_i$, then~\eqref{ineq: markov-residual-derivative-bounds} with $u=y$ and~\eqref{ineq: markov-G-slope-linear} imply that, for every $z\in[x\vee\theta_i,y]$,
    \begin{equation*}
        G_{\theta,i}'(z)\ge G_{\theta,i}'(\theta_i)-\frac{\gamma\mathcal R(\theta)}{\rho_{\min}\nu_{\min}(1-\gamma)^2}\ge-\frac{\mathcal R(\theta)}{\rho_{\min}\nu_{\min}(1-\gamma)^2}.
    \end{equation*}
    If $y<x\vee\theta_i$, then feasibility gives $y\ge x$ and hence $x\vee\theta_i=\theta_i$. In this case,~\eqref{ineq: markov-residual-derivative-bounds} with $u=y$ and~\eqref{ineq: markov-G-slope-linear} imply that, for every $z\in[y,\theta_i]$,
    \begin{equation*}
        G_{\theta,i}'(z)\le G_{\theta,i}'(\theta_i)+\frac{\gamma\mathcal R(\theta)}{\rho_{\min}\nu_{\min}(1-\gamma)^2}\le\frac{\mathcal R(\theta)}{\rho_{\min}\nu_{\min}(1-\gamma)^2}.
    \end{equation*}
    Integrating in the two cases and using $|y-(x\vee\theta_i)|\le B$ gives
    \begin{equation}\label{ineq: markov-one-step-linear}
        G_{\theta,i}(x\vee\theta_i)-G_{\theta,i}(y)\le\frac{B}{\rho_{\min}\nu_{\min}(1-\gamma)^2}\mathcal R(\theta).
    \end{equation}
    The same inequality holds when $y=x\vee\theta_i$ and, by the preceding reduction, when $y<0$. Therefore,~\eqref{ineq: markov-one-step-linear} holds at every state and for every feasible action. By the Bellman equation,
    \begin{equation*}
        \begin{aligned}
            J_{\pi_\theta}(x,i)-(TJ_{\pi_\theta})(x,i)
            &=G_{\theta,i}(x\vee\theta_i)-\min_{y\in[x,B]}G_{\theta,i}(y) \le\frac{B}{\rho_{\min}\nu_{\min}(1-\gamma)^2}\mathcal R(\theta), \qquad \forall(x,i)\in\cs.
        \end{aligned}
    \end{equation*}
    Since $\theta$ was arbitrary, Condition~\eqref{condition: general-mdp-pointwise-error} holds with $\alpha=1$ and $C=B/[\rho_{\min}\nu_{\min}(1-\gamma)^2]$. Theorem~\ref{thm: general-mdp-pointwise-PLK} gives~\eqref{ineq: markov-PLK-linear}. At any stationary point $\theta$ of $l$ over $\Theta$, $\mathcal R(\theta)=0$, so~\eqref{ineq: markov-PLK-linear} gives $l(\theta)=l(\theta^*)$.

    We next prove part~(ii). Suppose Assumption~\ref{assumption: inventory}.\ref{assumption: inventory-density-lower-bound} holds and $h+b>0$. Fix $i\in\ci$, $x\le B$, and a feasible action $y\in[x,B]$. As in part~(i), we first consider $y\in[0,B]$.

    If $y>x\vee\theta_i$, then~\eqref{ineq: markov-residual-derivative-bounds} with $u=y$ and~\eqref{ineq: markov-G-slope-quadratic} imply that, for every $z\in[x\vee\theta_i,y]$,
    \begin{equation*}
        \begin{aligned}
            G_{\theta,i}'(z)
            &\ge G_{\theta,i}'(\theta_i)+(h+b)\mu_D(z-\theta_i)-\frac{\gamma\mathcal R(\theta)}{\rho_{\min}\nu_{\min}(1-\gamma)^2}\\
            &\ge-\frac{\mathcal R(\theta)}{\rho_{\min}\nu_{\min}(1-\gamma)^2}+(h+b)\mu_D\left(z-(x\vee\theta_i)\right),
        \end{aligned}
    \end{equation*}
    where the last inequality uses $x\vee\theta_i\ge\theta_i$. If $y<x\vee\theta_i$, then $x\vee\theta_i=\theta_i$, and~\eqref{ineq: markov-residual-derivative-bounds} with $u=y$ and~\eqref{ineq: markov-G-slope-quadratic} imply that, for every $z\in[y,\theta_i]$,
    \begin{equation*}
        \begin{aligned}
            G_{\theta,i}'(z)
            &\le G_{\theta,i}'(\theta_i)-(h+b)\mu_D(\theta_i-z)+\frac{\gamma\mathcal R(\theta)}{\rho_{\min}\nu_{\min}(1-\gamma)^2} \le\frac{\mathcal R(\theta)}{\rho_{\min}\nu_{\min}(1-\gamma)^2}-(h+b)\mu_D(\theta_i-z).
        \end{aligned}
    \end{equation*}
    Integrating in the two cases and completing the square give
    \begin{equation*}
        \begin{aligned}
            G_{\theta,i}(y)-G_{\theta,i}(x\vee\theta_i) & \ge-\frac{\mathcal R(\theta)}{\rho_{\min}\nu_{\min}(1-\gamma)^2}\left|y-(x\vee\theta_i)\right|+\frac{(h+b)\mu_D}{2}\left|y-(x\vee\theta_i)\right|^2\\
            & \ge-\frac{\mathcal R(\theta)^2}{2(h+b)\mu_D\rho_{\min}^2\nu_{\min}^2(1-\gamma)^4}.
        \end{aligned}
    \end{equation*}
    The bound also holds when $y=x\vee\theta_i$. If $y<0$, then $G_{\theta,i}(y)\ge G_{\theta,i}(0)$, so applying the bound with $y=0$ gives the same conclusion. Hence, at every state and for every feasible action,
    \begin{equation}\label{ineq: markov-one-step-quadratic}
        G_{\theta,i}(x\vee\theta_i)-G_{\theta,i}(y)\le\frac{\mathcal R(\theta)^2}{2(h+b)\mu_D\rho_{\min}^2\nu_{\min}^2(1-\gamma)^4}.
    \end{equation}
    By the Bellman equation,
    \begin{equation*}
        \begin{aligned}
            J_{\pi_\theta}(x,i)-(TJ_{\pi_\theta})(x,i)
            &=G_{\theta,i}(x\vee\theta_i)-\min_{y\in[x,B]}G_{\theta,i}(y) \le\frac{\mathcal R(\theta)^2}{2(h+b)\mu_D\rho_{\min}^2\nu_{\min}^2(1-\gamma)^4}, \qquad \forall(x,i)\in\cs.
        \end{aligned}
    \end{equation*}
    Thus, Condition~\eqref{condition: general-mdp-pointwise-error} holds with $\alpha=2$ and $C=1/[2(h+b)\mu_D\rho_{\min}^2\nu_{\min}^2(1-\gamma)^4]$. Theorem~\ref{thm: general-mdp-pointwise-PLK} gives~\eqref{ineq: markov-PLK}. \Halmos
\end{proof}

\end{APPENDICES}

\ACKNOWLEDGMENT{During the preparation of this manuscript, the authors used ChatGPT Pro to develop examples, assist with the analysis, review mathematical arguments, and refine the exposition. The authors independently verified all results.}

\bibliographystyle{informs2014}
\bibliography{ref}

\setcounter{nowappendix}{0}
\ECSwitch

\renewcommand{\theHsection}{EC.\arabic{section}}
\renewcommand{\theHequation}{EC.\arabic{equation}}

\providecommand{\theHlemma}{}
\renewcommand{\theHlemma}{EC.\arabic{lemma}}

\ECHead{Electronic Companion to ``Benign Nonconvex Landscape for Policy Optimization: Infinite-Horizon Discounted MDPs with General State and Action Spaces''}

\section{Omitted Proofs in Section~\ref{sec: PLK condition}} \label{appendix: PLK condition}

\begin{proof}{Proof of Lemma~\ref{lemma: gradient dominance and PLK}}
    Let $f^*\coloneqq\min_{y\in\cx}f(y)$. Fix $x\in\cx$ and $g\in N_\cx(x)$. Since $\langle g,x-y\rangle\ge0$ for every $y\in\cx$, gradient dominance gives
    \begin{equation*}
        \begin{aligned}
            f(x)-f^*
            &\le \sup_{y\in\cx}\left\{c\langle\nabla f(x),x-y\rangle-\frac{\mu}{2}\|x-y\|_2^2\right\}\\
            &\le \sup_{y\in\cx}\left\{c\langle\nabla f(x)+g,x-y\rangle-\frac{\mu}{2}\|x-y\|_2^2\right\}\\
            &\le \sup_{y\in\cx}\left\{c\|\nabla f(x)+g\|_2\|x-y\|_2-\frac{\mu}{2}\|x-y\|_2^2\right\},
        \end{aligned}
    \end{equation*}
    where the last inequality follows from the Cauchy--Schwarz inequality.

    Suppose first that $\mu>0$. For every $y\in\cx$,
    \begin{equation*}
        \begin{aligned}
            c\|\nabla f(x)+g\|_2\|x-y\|_2-\frac{\mu}{2}\|x-y\|_2^2
            &= \frac{c^2}{2\mu}\|\nabla f(x)+g\|_2^2 -\frac{\mu}{2}\left(\|x-y\|_2-\frac{c}{\mu}\|\nabla f(x)+g\|_2\right)^2\\
            &\le \frac{c^2}{2\mu}\|\nabla f(x)+g\|_2^2.
        \end{aligned}
    \end{equation*}
    Since the upper bound is independent of $y$, taking the supremum over $y\in\cx$ and then the minimum over $g\in N_\cx(x)$ gives
    \begin{equation*}
        f(x)-f^*\le\frac{c^2}{2\mu}\min_{g\in N_\cx(x)}\|\nabla f(x)+g\|_2^2,
    \end{equation*}
    which proves the P{\L}K condition with exponent $2$ and constant $\mu/c^2$.

    Suppose next that $\mu=0$ and $\cx$ has diameter at most $R$. Since $\|x-y\|_2\le R$ for every $y\in\cx$, the same bound yields
    \begin{equation*}
        f(x)-f^*\le cR\|\nabla f(x)+g\|_2.
    \end{equation*}
    Taking the minimum over $g\in N_\cx(x)$ gives
    \begin{equation*}
        f(x)-f^*\le cR\min_{g\in N_\cx(x)}\|\nabla f(x)+g\|_2,
    \end{equation*}
    which proves the P{\L}K condition with exponent $1$ and constant $1/(2Rc)$. \Halmos
\end{proof}

\begin{proof}{Proof of Lemma~\ref{lemma: convergence rate}}
    Let $\Delta_k=f(x_k)-f^*$. With the stepsize $\gamma_k=1/L$, the projection optimality condition gives
    \begin{equation*}
        \left\langle x_{k+1}-x_k+\frac{1}{L}\nabla f(x_k), x-x_{k+1}\right\rangle \ge 0, \qquad \forall x\in\cx .
    \end{equation*}
    Taking $x=x_k$, we obtain
    \begin{equation*}
        \left\langle \nabla f(x_k), x_{k+1}-x_k \right\rangle
        \le -L\|x_{k+1}-x_k\|_2^2 .
    \end{equation*}
    Therefore, by the $L$-smoothness of $f$,
    \begin{equation}\label{ineq: pgd descent}
        f(x_{k+1})-f(x_k) \le -\frac{L}{2}\|x_{k+1}-x_k\|_2^2 .
    \end{equation}

    The same projection optimality condition also implies
    \begin{equation*}
        L(x_k-x_{k+1})-\nabla f(x_k) \in N_\cx(x_{k+1}) .
    \end{equation*}
    Hence,
    \begin{equation}\label{ineq: pgd residual}
        \min_{g\in N_\cx(x_{k+1})} \left\| \nabla f(x_{k+1})+g \right\|_2 \le \left\|L(x_k-x_{k+1})+\nabla f(x_{k+1})-\nabla f(x_k)\right\|_2 \le 2L\|x_{k+1}-x_k\|_2 ,
    \end{equation}
    where the last inequality follows from the $L$-Lipschitz continuity of $\nabla f$.

    We first consider the case where the P{\L}K exponent is $1$. By the P{\L}K condition and \eqref{ineq: pgd residual},
    \begin{equation*}
        \Delta_{k+1} \le \frac{L}{\mu}\|x_{k+1}-x_k\|_2 .
    \end{equation*}
    Combining this inequality with \eqref{ineq: pgd descent} yields
    \begin{equation}\label{ineq: recursion alpha one}
        \Delta_k-\Delta_{k+1} \ge \frac{\mu^2}{2L}\Delta_{k+1}^2.
    \end{equation}
    If $\Delta_j=0$ for some $j$, then the desired bound holds trivially for all $k\ge j$. Otherwise, all reciprocals below are well-defined. From \eqref{ineq: recursion alpha one},
    \begin{equation*}
        \Delta_k \ge \Delta_{k+1}\left(1+\frac{\mu^2}{2L}\Delta_{k+1}\right),
    \end{equation*}
    and therefore
    \begin{equation*}
        \frac{1}{\Delta_{k+1}}-\frac{1}{\Delta_k} \ge \frac{\mu^2/(2L)}{1+\mu^2\Delta_{k+1}/(2L)} \ge \frac{\mu^2}{2L+\mu^2\Delta_0},
    \end{equation*}
    where the last inequality uses $\Delta_{k+1}\le \Delta_0$. Summing this inequality from $0$ to $k-1$ gives
    \begin{equation*}
        \Delta_k \le \left( \frac{1}{\Delta_0} + \frac{k\mu^2} {2L+\mu^2\Delta_0} \right)^{-1} \le \frac{2L/\mu^2+\Delta_0}{k+1}.
    \end{equation*}
    This proves part \ref{alg: PGD 1-KL}.

    We next consider the case where the P{\L}K exponent is $2$. By the P{\L}K condition and \eqref{ineq: pgd residual},
    \begin{equation*}
        \Delta_{k+1} \le \frac{1}{2\mu}\left(2L\|x_{k+1}-x_k\|_2\right)^2 = \frac{2L^2}{\mu}\|x_{k+1}-x_k\|_2^2 .
    \end{equation*}
    Combining this inequality with \eqref{ineq: pgd descent} gives
    \begin{equation*}
        \Delta_k-\Delta_{k+1} \ge \frac{L}{2}\|x_{k+1}-x_k\|_2^2 \ge \frac{\mu}{4L}\Delta_{k+1}.
    \end{equation*}
    Thus,
    \begin{equation*}
        \Delta_{k+1} \le \frac{4L}{4L+\mu}\Delta_k = \left(1-\frac{\mu}{4L+\mu}\right)\Delta_k .
    \end{equation*}
    Iterating the above inequality yields
    \begin{equation*}
        f(x_k)-f^* \le \left(1-\frac{\mu}{4L+\mu}\right)^k \left[f(x_0)-f^*\right].
    \end{equation*}
    This proves part \ref{alg: PGD 2-KL}. \Halmos
\end{proof}

\section{Technical Proofs for the Applications}\label{appendix: example}

\subsection{Inventory System with Markov-Modulated Demand}\label{appendix: inventory}

\begin{lemma} \label{lemma: markov-demand-regularity}
    Suppose the first three parts of Assumption~\ref{assumption: inventory} hold. Then the inventory model with the state-dependent base-stock policy class $\Pi_\Theta$ and initial distribution $\rho\otimes\nu$ satisfies Assumptions~\ref{assumption: general-mdp-technical} and~\ref{assumption: general-mdp-regularity}.
\end{lemma}

\begin{proof}{Proof of Lemma~\ref{lemma: markov-demand-regularity}}
    We verify the two parts of Assumption~\ref{assumption: general-mdp-technical} and the two parts of Assumption~\ref{assumption: general-mdp-regularity}.

    \noindent \textbf{Assumption~\ref{assumption: general-mdp-technical}.\ref{assumption: general-mdp-measurability}:} The state space $\cs=(-\infty,B]\times\ci$, the action space $\ca=(-\infty,B]$, and the feasible state-action set $\Gamma=\{((x,i),y)\in\cs\times\ca:x\le y\le B\}$ are Borel sets. For each $i\in\ci$, the finite conditional demand mean ensures that $C_i$ is real-valued, and the holding and backlogging cost implies that $C_i$ is $\max\{h,b\}$-Lipschitz continuous. Thus, $g((x,i),y)=C_i(y)$ is measurable on $\Gamma$. For every Borel set $\cm\subseteq\cs$,
    \begin{equation*}
        P(\cm\mid(x,i),y)=\sum_{j\in\ci}p_{ij}\mathbb E_{D\sim P_D(\cdot\mid i)}\left[\mathbf 1_{\cm}(y-D,j)\right].
    \end{equation*}
    Since demand is nonnegative, this defines a probability measure on $\cs$ for every feasible state-action pair. Moreover, the mapping $(y,d)\mapsto(y-d,j)$ is continuous for each $j$, so $((x,i),y)\mapsto P(\cm\mid(x,i),y)$ is measurable. Hence, $P$ is a stochastic kernel on $\cs$ given $\Gamma$.
    
    Fix $\theta\in\Theta$. The policy $\pi_\theta(x,i)=x\vee\theta_i$ is measurable and feasible. Since its post-order inventory always belongs to $[0,B]$,
    \begin{equation*}
        0\le J_{\pi_\theta}(x,i)\le\frac{1}{1-\gamma}\max_{j\in\ci}\sup_{y\in[0,B]}C_j(y)<\infty, \qquad (x,i)\in\cs.
    \end{equation*}
    Couple two inventory trajectories starting from $(x,i)$ and $(x',i)$ under $\pi_\theta$ using the same environment process and demand sequence. The recursion $x_{t+1}=x_t\vee\theta_{i_t}-D_t$ implies that $|x_t-x_t'|\le|x-x'|$ for every $t$. The Lipschitz continuity of $C_i$ therefore gives
    \begin{equation*}
        \left|J_{\pi_\theta}(x,i)-J_{\pi_\theta}(x',i)\right|\le\max\{h,b\}\sum_{t=0}^{\infty}\gamma^t|x-x'|=\frac{\max\{h,b\}}{1-\gamma}|x-x'|.
    \end{equation*}
    Thus, $J_{\pi_\theta}(\cdot,i)$ is Lipschitz continuous for every $i\in\ci$, and $J_{\pi_\theta}$ is real-valued and measurable on $\cs$. Conditioning on the first transition yields $J_{\pi_\theta}=T_{\pi_\theta}J_{\pi_\theta}$.
    
    The boundedness and continuity of $J_{\pi_\theta}$ imply, by dominated convergence, that $G_{\theta,i}$ is real-valued and continuous on $(-\infty,B]$ for every $i\in\ci$. Consequently, $Q_{\pi_\theta}((x,i),y)=G_{\theta,i}(y)$ is real-valued and measurable on $\Gamma$. Since $[x,B]$ is nonempty and compact, $G_{\theta,i}$ attains its minimum over $[x,B]$. Let $\pi'(x,i)$ be its smallest minimizer. For fixed $i$ and $x<x'\le B$, if $\pi'(x,i)\ge x'$, then $\pi'(x',i)=\pi'(x,i)$; otherwise, $\pi'(x',i)\ge x'>\pi'(x,i)$. Hence, $\pi'(\cdot,i)$ is nondecreasing and therefore measurable. Since $\pi'$ is also feasible, $\pi'\in\Pi$, and
    \begin{equation*}
        Q_{\pi_\theta}((x,i),\pi'(x,i))=G_{\theta,i}(\pi'(x,i))=\min_{y\in[x,B]}G_{\theta,i}(y)=(TJ_{\pi_\theta})(x,i).
    \end{equation*}
    In particular, $TJ_{\pi_\theta}$ is real-valued and measurable on $\cs$. This verifies Assumption~\ref{assumption: general-mdp-technical}.\ref{assumption: general-mdp-measurability}.

    \noindent \textbf{Assumption~\ref{assumption: general-mdp-technical}.\ref{assumption: general-mdp-integrability}:} Fix $\theta,\theta',\bar\theta\in\Theta$. The preceding verification shows that $J_{\pi_\theta}$ is nonnegative and bounded. Since $\pi_{\theta'}(x,i)\in[0,B]$ for every $(x,i)\in\cs$,
    \begin{equation*}
        0\le(TJ_{\pi_\theta})(x,i)\le(T_{\pi_{\theta'}}J_{\pi_\theta})(x,i)\le\frac{1}{1-\gamma}\max_{j\in\ci}\sup_{y\in[0,B]}C_j(y)<\infty.
    \end{equation*}
    Thus, $J_{\pi_\theta}$, $T_{\pi_{\theta'}}J_{\pi_\theta}$, and $TJ_{\pi_\theta}$ are bounded and therefore integrable with respect to the probability measure $\eta_{\pi_{\bar\theta}}$.
    
    Next, let $\pi'\in\Pi$ be arbitrary. Since demand is nonnegative, $C_i(y)=C_i(0)-by$ for $y<0$. The feasibility condition $x\le\pi'(x,i)\le B$ and the bound on $J_{\pi_\theta}$ therefore imply that
    \begin{equation*}
        0\le(T_{\pi'}J_{\pi_\theta})(x,i)\le\frac{1}{1-\gamma}\max_{j\in\ci}\sup_{y\in[0,B]}C_j(y)+b(-x)^+.
    \end{equation*}
    Under $\pi_\theta$, the post-order inventory is nonnegative, so $x_t=(x_{t-1}\vee\theta_{i_{t-1}})-D_{t-1}\ge-D_{t-1}$ and hence $(-x_t)^+\le D_{t-1}$ for every $t\ge1$. Consequently,
    \begin{equation*}
        \begin{aligned}
            \int_{\cs}(-x)^+\,\eta_{\pi_\theta}(dx,di)
            &=(1-\gamma)\sum_{t=0}^{\infty}\gamma^t\mathbb E_{\rho\otimes\nu}^{\pi_\theta}\left[(-x_t)^+\right] \le(1-\gamma)\mathbb E_\rho\left[(-x_0)^+\right]+\gamma\max_{j\in\ci}\mathbb E_{D\sim P_D(\cdot\mid j)}[D]<\infty,
        \end{aligned}
    \end{equation*}
    where finiteness follows from the finite means of the initial inventory and the conditional demand distributions, together with the finiteness of $\ci$. Thus, $T_{\pi'}J_{\pi_\theta}$ is integrable with respect to $\eta_{\pi_\theta}$ for every $\pi'\in\Pi$. This verifies Assumption~\ref{assumption: general-mdp-technical}.\ref{assumption: general-mdp-integrability}.

    \noindent \textbf{Assumption~\ref{assumption: general-mdp-regularity}.\ref{assumption: optimal policy}:} By \citet{song1993inventory}, the inventory system with Markov-modulated demand admits an optimal policy of the form $\pi_{\theta^*}(x, i)=x\vee\theta_i^*$ for some $\theta^*\in\Theta$. Hence, $\Pi_\Theta$ contains an optimal policy.

    \noindent \textbf{Assumption~\ref{assumption: general-mdp-regularity}.\ref{assumption: differentiablity}:} Fix $\theta\in\Theta$. The preceding measurability verification shows that $J_{\pi_\theta}(\cdot,i)$ is $\max\{h,b\}/(1-\gamma)$-Lipschitz continuous for every $i\in\ci$. Its derivative therefore exists almost everywhere and is bounded by the same constant. Since each conditional demand distribution admits a density, $C_i'(y)=(h+b)P_D(y\mid i)-b$, and dominated convergence gives
    \begin{equation*}
        G_{\theta,i}'(y)=C_i'(y)+\gamma\sum_{j\in\ci}p_{ij}\mathbb E_{D\sim P_D(\cdot\mid i)}\left[\partial_xJ_{\pi_\theta}(y-D,j)\right], \qquad y<B,
    \end{equation*}
    where $\partial_xJ_{\pi_\theta}$ is assigned zero wherever the derivative does not exist. Each expectation can be written as $\int_{-\infty}^{B}\partial_xJ_{\pi_\theta}(x,j)p_D(y-x\mid i)\,dx$. The boundedness of $\partial_xJ_{\pi_\theta}$ and continuity of translations in $L^1$ imply that this expression is continuous in $y$. Thus, $G_{\theta,i}$ is continuously differentiable on $(-\infty,B]$, with $G_{\theta,i}'(B)$ denoting its left derivative, and
    \begin{equation*}
        \sup_{y\le B}|G_{\theta,i}'(y)|\le\frac{\max\{h,b\}}{1-\gamma}, \qquad i\in\ci.
    \end{equation*}
    To handle boundary parameters, extend $G_{\theta,i}$ to $y>B$ by setting $G_{\theta,i}(y)=G_{\theta,i}(B)+G_{\theta,i}'(B)(y-B)$, retaining the same notation. This extension is continuously differentiable on $\R$ and satisfies the same derivative bound.
    
    Construct the inventory trajectories for all $\bar\theta\in(-B,2B)^{|\ci|}$ using the same initial state, environment process, and demand sequence. Write $x_0(\bar\theta)=x_0$ and $x_{t+1}(\bar\theta)=x_t(\bar\theta)\vee\bar\theta_{i_t}-D_t$. For $\theta',\bar\theta\in\Theta$, the definition of the discounted occupancy measure gives
    \begin{equation*}
        \cb(\pi_{\theta'}\mid\eta_{\pi_{\bar\theta}},J_{\pi_\theta})=(1-\gamma)\sum_{t=0}^{\infty}\gamma^t\mathbb E\left[G_{\theta,i_t}\left(x_t(\bar\theta)\vee\theta'_{i_t}\right)\right],
    \end{equation*}
    where the expectation is with respect to the common initial state, environment process, and demand sequence. The right-hand side defines an extension to $\theta',\bar\theta\in(-B,2B)^{|\ci|}$. Indeed, nonnegative demand and $x_0\le B$ imply that $x_t(\bar\theta)\vee\theta'_{i_t}\in[-B,2B]$, so the summands are uniformly bounded before discounting.
    
    The inventory recursion implies that $\bar\theta\mapsto x_t(\bar\theta)$ is $1$-Lipschitz with respect to the $\ell_\infty$ norm for every $t$. Fix $(\theta',\bar\theta)$ and $t$. The densities of the initial inventory and conditional demands ensure that the events $x_s(\bar\theta)=\bar\theta_{i_s}$ for some $0\le s<t$, or $x_t(\bar\theta)=\theta'_{i_t}$, have probability zero. Outside these events, the trajectory recursion is differentiable in a neighborhood of $\bar\theta$. For every $k\in\ci$, starting from $\partial_{\bar\theta_k}x_0(\bar\theta)=0$, we obtain
    \begin{equation*}
        \partial_{\bar\theta_k}x_{s+1}(\bar\theta)=\mathbf 1_{\{x_s(\bar\theta)>\bar\theta_{i_s}\}}\partial_{\bar\theta_k}x_s(\bar\theta)+\mathbf 1_{\{x_s(\bar\theta)<\bar\theta_{i_s}\}}\mathbf 1_{\{i_s=k\}}, \qquad 0\le s<t.
    \end{equation*}
    Consequently, $|\partial_{\bar\theta_k}x_t(\bar\theta)|\le1$ almost surely. Differentiating the integrand in the preceding series gives
    \begin{equation*}
        \begin{aligned}
            \partial_{\theta'_k}G_{\theta,i_t}\left(x_t(\bar\theta)\vee\theta'_{i_t}\right)&=\mathbf 1_{\{i_t=k,\ x_t(\bar\theta)<\theta'_k\}}G_{\theta,k}'(\theta'_k),\\
            \partial_{\bar\theta_k}G_{\theta,i_t}\left(x_t(\bar\theta)\vee\theta'_{i_t}\right)&=\mathbf 1_{\{x_t(\bar\theta)>\theta'_{i_t}\}}G_{\theta,i_t}'(x_t(\bar\theta))\partial_{\bar\theta_k}x_t(\bar\theta).
        \end{aligned}
    \end{equation*}
    Both derivatives are bounded in absolute value by $\max\{h,b\}/(1-\gamma)$, uniformly in $t$ and the parameters. The Lipschitz bounds provide the same bound for the corresponding difference quotients. Dominated convergence therefore permits differentiation under the expectation.
    
    To establish joint continuity of these derivatives, consider any sequence of parameter pairs converging to $(\theta',\bar\theta)$. For each fixed $t$, outside the zero-probability events described above, the finitely many strict comparisons in the trajectory recursion and the final maximum remain unchanged for all sufficiently large indices. The trajectory derivatives are then unchanged, and continuity of $G_{\theta,i}'$ implies convergence of both displayed derivatives. Dominated convergence shows that the partial derivatives of each expected summand are jointly continuous. The uniform bounds on the summands and their derivatives, together with the geometric discount weights, imply uniform convergence of the series and its partial derivative series. Hence, the extended mapping is jointly continuously differentiable on $(-B,2B)^{|\ci|}\times(-B,2B)^{|\ci|}$, an open set containing $\Theta\times\Theta$. In particular, for $\theta',\bar\theta\in\Theta$,
    \begin{equation*}
        \partial_{\theta'_k}\cb(\pi_{\theta'}\mid\eta_{\pi_{\bar\theta}},J_{\pi_\theta})=\eta_{\pi_{\bar\theta}}\left((-\infty,\theta'_k)\times\{k\}\right)G_{\theta,k}'(\theta'_k), \qquad k\in\ci.
    \end{equation*}
    This verifies Assumption~\ref{assumption: general-mdp-regularity}.\ref{assumption: differentiablity}. \Halmos
\end{proof}

\begin{lemma}\label{lemma: markov-demand-smoothness}
    Suppose the first three parts of Assumption~\ref{assumption: inventory} hold. Then $\nabla l$ is Lipschitz continuous on $\Theta$ with respect to the Euclidean norm.
\end{lemma}

\begin{proof}{Proof of Lemma~\ref{lemma: markov-demand-smoothness}}
    Use the inventory trajectories $x_t(\theta)$ constructed in the proof of Lemma~\ref{lemma: markov-demand-regularity} with the same initial state, environment process, and demand sequence, and write $y_t(\theta)=x_t(\theta)\vee\theta_{i_t}$. These trajectories are defined for $\theta\in(-B,2B)^{|\ci|}$ and satisfy $y_t(\theta)\in[-B,2B]$. At each fixed parameter, the threshold-equality events have probability zero, and the derivative recursion established there shows that $\nabla y_t(\theta)$ is either zero or a coordinate vector. In particular, $\|\nabla y_t(\theta)\|_2\le1$ almost surely. Moreover, $C_i'(y)=(h+b)P_D(y\mid i)-b$, so $|C_i'(y)|\le\max\{h,b\}$. The same dominated-convergence and uniform-convergence arguments therefore show that the discounted cost series defines a continuously differentiable extension of $l$ to $(-B,2B)^{|\ci|}$, with
    \begin{equation}\label{eq: markov-pathwise-gradient}
        \nabla l(\theta)=(1-\gamma)\sum_{t=0}^{\infty}\gamma^t\mathbb E\left[C_{i_t}'(y_t(\theta))\nabla y_t(\theta)\right], \qquad \theta\in\Theta.
    \end{equation}
    In particular, this representation also holds at boundary parameters.

    Fix $\theta,\theta'\in\Theta$. Since $|\theta_i-\theta_i'|\le\|\theta-\theta'\|_2$ for every $i\in\ci$, the maximum operation and induction on the inventory recursion give
    \begin{equation*}
        |x_t(\theta)-x_t(\theta')|\le\|\theta-\theta'\|_2, \qquad |y_t(\theta)-y_t(\theta')|\le\|\theta-\theta'\|_2, \qquad t\ge0.
    \end{equation*}
    If the two trajectories differ in whether to place an order at time $s$, then
    \begin{equation*}
        |x_s(\theta)-\theta_{i_s}|\le|x_s(\theta)-x_s(\theta')|+|\theta_{i_s}-\theta_{i_s}'|\le2\|\theta-\theta'\|_2.
    \end{equation*}
    At $s=0$, independence of $x_0$ and $i_0$ and the density bound for $x_0$ imply
    \begin{equation*}
        \mathbb P\left(|x_0-\theta_{i_0}|\le2\|\theta-\theta'\|_2\right)\le4L_\rho\|\theta-\theta'\|_2.
    \end{equation*}
    For $s\ge1$, condition on the history before $D_{s-1}$ is drawn and on $i_s$. By the conditional independence of demand and the next environment state, $D_{s-1}$ retains its conditional density bounded by $L_D$. Since $x_s(\theta)=y_{s-1}(\theta)-D_{s-1}$, it follows that
    \begin{equation*}
        \mathbb P\left(|x_s(\theta)-\theta_{i_s}|\le2\|\theta-\theta'\|_2\right)\le4L_D\|\theta-\theta'\|_2, \qquad s\ge1.
    \end{equation*}

    Whenever the two trajectories agree on whether to order in every period $s=0,\ldots,t$, the derivative recursion gives $\nabla y_t(\theta)=\nabla y_t(\theta')$. Since each gradient has norm at most one, the union bound yields
    \begin{equation*}
        \mathbb E\left[\|\nabla y_t(\theta)-\nabla y_t(\theta')\|_2\right] \le2\sum_{s=0}^{t}\mathbb P\left(|x_s(\theta)-\theta_{i_s}|\le2\|\theta-\theta'\|_2\right) \le8(L_\rho+tL_D)\|\theta-\theta'\|_2.
    \end{equation*}
    Each $C_i'$ is $(h+b)L_D$-Lipschitz continuous. Combining the preceding bounds with~\eqref{eq: markov-pathwise-gradient} therefore gives
    \begin{equation*}
        \begin{aligned}
            \|\nabla l(\theta)-\nabla l(\theta')\|_2
            &\le(1-\gamma)\sum_{t=0}^{\infty}\gamma^t\left[(h+b)L_D+8\max\{h,b\}(L_\rho+tL_D)\right]\|\theta-\theta'\|_2\\
            &=\left[(h+b)L_D+8\max\{h,b\}\left(L_\rho+\frac{\gamma L_D}{1-\gamma}\right)\right]\|\theta-\theta'\|_2.
        \end{aligned}
    \end{equation*}
    Thus, $\nabla l$ is Lipschitz continuous on $\Theta$ with respect to the Euclidean norm. \Halmos
\end{proof}

\subsection{Stochastic Cash-Balance Problem}\label{appendix: cash-balance}

\begin{lemma}\label{lemma: cash-balance-regularity}
    Suppose the first three parts of Assumption~\ref{assumption: cash balance} hold. Then the stochastic cash-balance model with the two-sided base-stock policy class $\Pi_\Theta$ and initial distribution $\rho$ satisfies Assumptions~\ref{assumption: general-mdp-technical} and~\ref{assumption: general-mdp-regularity}.
\end{lemma}

\begin{proof}{Proof of Lemma~\ref{lemma: cash-balance-regularity}}
    We follow the proof of Lemma~\ref{lemma: markov-demand-regularity} and record the modifications needed for two-sided adjustments and transaction costs. The bounded action interval and the finite absolute first moments of $D$ and $x_0$ imply that every $\eta_{\pi_{\bar\theta}}$ has a finite first moment and that, for some constant $K<\infty$ independent of $\theta$,
    \begin{equation*}
        |J_{\pi_\theta}(x)|+\sup_{y\in[\underline B,\bar B]}|Q_{\pi_\theta}(x,y)|\le K(1+|x|), \qquad x\in\R,\quad\theta\in\Theta.
    \end{equation*}
    These bounds give all the required absolute-integrability properties, including those involving arbitrary $\pi'\in\Pi$. The same coupling argument shows that $J_{\pi_\theta}$ is Lipschitz continuous. Consequently, $Q_{\pi_\theta}$ is continuous, and minimization over the fixed compact action interval admits a measurable minimizing selector. Together with the measurability of the cost and transition kernel and the Bellman identity, these observations verify Assumption~\ref{assumption: general-mdp-technical}.

    The choice of $\underline B$ and $\bar B$ ensures that $\Pi_\Theta$ contains an optimal policy \citep{whisler1967stochastic,eppen1969cash}. Thus, Assumption~\ref{assumption: general-mdp-regularity}.\ref{assumption: optimal policy} holds. 
    
    For differentiability, fix $\theta\in\Theta$. The convolution argument in the proof of Lemma~\ref{lemma: markov-demand-regularity} shows that the defining formula~\eqref{eq: cash-G} extends $G_\theta$ to a continuously differentiable function on $\R$ with bounded derivative. To handle coincident thresholds, define, for arbitrary threshold pairs,
    \begin{equation*}
        \widetilde\pi_{\theta'}(x)\coloneqq x+(\underline\theta'-x)^+-(x-\bar\theta')^+.
    \end{equation*}
    This agrees with $\pi_{\theta'}$ on $\Theta$. Using a common initial state and demand sequence, construct $x_0(\bar\theta)=x_0$ and $x_{t+1}(\bar\theta)=\widetilde\pi_{\bar\theta}(x_t(\bar\theta))-D_t$. The series
    \begin{equation*}
        (1-\gamma)\sum_{t=0}^{\infty}\gamma^t\E\Big[ k(\underline\theta'-x_t(\bar\theta))^+ + q(x_t(\bar\theta)-\bar\theta')^+ +G_\theta\big(\widetilde\pi_{\theta'}(x_t(\bar\theta))\big)\Big]
    \end{equation*}
    agrees with $\cb(\pi_{\theta'}\mid\eta_{\pi_{\bar\theta}}, J_{\pi_\theta})$ on $\Theta\times\Theta$ and defines an extension to a bounded open neighborhood of this set. The auxiliary policy is $1$-Lipschitz in the state and in each threshold, so each coordinate derivative of $x_t(\bar\theta)$ is bounded in absolute value by $t$. The initial-state and demand densities make all threshold equalities null events. Hence, the same common-noise and dominated-convergence arguments as in the inventory proof establish continuous differentiability of each expected summand. Their absolute values are uniformly bounded on the neighborhood, and their derivatives are bounded by a constant times $1+t$. Since $\sum_{t=0}^{\infty}\gamma^t(1+t)<\infty$, the series defines a jointly continuously differentiable extension. This verifies Assumption~\ref{assumption: general-mdp-regularity}.\ref{assumption: differentiablity}. \Halmos
\end{proof}

\begin{proof}{Proof of Lemma~\ref{lemma: cash-approx-convexity}}
    Fix $\theta=(\underline\theta,\bar\theta)\in\Theta$. The proof follows the argument in Lemma~\ref{lemma: markov-approx-convexity}; we record only the modifications caused by the two thresholds and the transaction costs.

    Since $\eta_{\pi_\theta}\succeq(1-\gamma)\rho$, Assumption~\ref{assumption: cash balance}.\ref{assumption: cash balance-exploratory-initial-distribution} gives uniform lower bounds of $(1-\gamma)\alpha$ for both $\eta_{\pi_\theta}((-\infty,u))$ and $\eta_{\pi_\theta}((u,\infty))$, for every $u\in[\underline B,\bar B]$. Therefore, the Policy Gradient Theorem and the normal-cone argument used in the proof of Lemma~\ref{lemma: markov-approx-convexity} yield
    \begin{equation}\label{ineq: cash-lower-feasible-direction}
        \left(k+G_\theta'(\underline\theta)\right)(u-\underline\theta) \ge -\frac{\mathcal R(\theta)}{\alpha(1-\gamma)}|u-\underline\theta|, \qquad u\in[\underline B,\bar\theta],
    \end{equation}
    and
    \begin{equation}\label{ineq: cash-upper-feasible-direction}
        \left(G_\theta'(\bar\theta)-q\right)(u-\bar\theta) \ge -\frac{\mathcal R(\theta)}{\alpha(1-\gamma)}|u-\bar\theta|, \qquad u\in[\underline\theta,\bar B].
    \end{equation}

    Let
    \begin{equation*}
        m_\theta \coloneqq \min\left\{ 0, \inf_{x<y} \left[ \partial_-J_{\pi_\theta}(y) - \partial_+J_{\pi_\theta}(x) \right] \right\}.
    \end{equation*}
    By Lemma~\ref{lemma: cash-balance-regularity}, $J_{\pi_\theta}$ is Lipschitz continuous, so $m_\theta>-\infty$. As in the proof of Lemma~\ref{lemma: markov-approx-convexity}, differentiation under the expectation in \eqref{eq: cash-G} and the convexity of $C$ give
    \begin{equation}\label{ineq: cash-G-bootstrap}
        G_\theta'(z_2)-G_\theta'(z_1) \ge \gamma m_\theta, \qquad \underline B\le z_1<z_2\le\bar B.
    \end{equation}

    By \eqref{eq: cash-bellman-J-G}, for every $x<y$,
    \begin{equation*}
        \begin{aligned}
            \partial_-J_{\pi_\theta}(y)-\partial_+J_{\pi_\theta}(x)
            &= G_\theta'((y\vee\underline\theta)\wedge\bar\theta) - G_\theta'((x\vee\underline\theta)\wedge\bar\theta) \\
            &\quad+ \left(\mathbf 1_{\{x<\underline\theta\}}-\mathbf 1_{\{y\le\underline\theta\}}\right) \left(k+G_\theta'(\underline\theta)\right) \\
            &\quad+ \left(\mathbf 1_{\{x\ge\bar\theta\}}-\mathbf 1_{\{y>\bar\theta\}}\right) \left(G_\theta'(\bar\theta)-q\right).
        \end{aligned}
    \end{equation*}
    If $(x,y)$ crosses only the lower threshold, \eqref{ineq: cash-lower-feasible-direction} with $u=\bar\theta$ gives
    \begin{equation*}
        k+G_\theta'(\underline\theta) \ge -\frac{\mathcal R(\theta)}{\alpha(1-\gamma)}.
    \end{equation*}
    If $(x,y)$ crosses only the upper threshold, \eqref{ineq: cash-upper-feasible-direction} with $u=\underline\theta$ gives
    \begin{equation*}
        G_\theta'(\bar\theta)-q \le \frac{\mathcal R(\theta)}{\alpha(1-\gamma)}.
    \end{equation*}
    If $(x,y)$ crosses both thresholds, the derivative difference equals $k+q\ge0$. Combining these observations with \eqref{ineq: cash-G-bootstrap} gives
    \begin{equation*}
        \partial_-J_{\pi_\theta}(y)-\partial_+J_{\pi_\theta}(x) \ge \gamma m_\theta - \frac{\mathcal R(\theta)}{\alpha(1-\gamma)}, \qquad x<y.
    \end{equation*}
    If the infimum in the definition of $m_\theta$ is nonnegative, the required intermediate bound is immediate. Otherwise, taking the infimum over $x<y$ and rearranging gives
    \begin{equation*}
        \partial_-J_{\pi_\theta}(y)-\partial_+J_{\pi_\theta}(x) \ge -\frac{\mathcal R(\theta)}{\alpha(1-\gamma)^2}, \qquad x<y.
    \end{equation*}

    Substituting this intermediate bound into the derivative representation of $G_\theta$ and using the convexity of $C$ proves \eqref{ineq: cash-G-slope-linear}. In particular, when $\mathcal R(\theta)=0$, $G_\theta'$ is nondecreasing, so $G_\theta$ is convex on $[\underline B,\bar B]$. Under Assumption~\ref{assumption: cash balance}.\ref{assumption: cash balance-density-lower-bound},
    \begin{equation*}
        C'(z_2)-C'(z_1) = (h+p)\mathbb P(z_1<D\le z_2) \ge (h+p)\mu_D(z_2-z_1).
    \end{equation*}
    Combining this inequality with the same intermediate bound proves \eqref{ineq: cash-G-slope-quadratic}. \Halmos
\end{proof}

\begin{proof}{Proof of Theorem~\ref{thm: cash-balance-PLK}}
    If $\underline B=\bar B$, then $\Theta$ is a singleton and all conclusions are immediate. Suppose $\underline B<\bar B$ and fix $\theta=(\underline\theta,\bar\theta)\in\Theta$. By Lemma~\ref{lemma: cash-balance-regularity}, Assumptions~\ref{assumption: general-mdp-technical} and~\ref{assumption: general-mdp-regularity} hold.

    We first prove part~(i). Fix $x\in\R$ and write $a=\pi_\theta(x)$. By~\eqref{eq: cash-Q-G}, the derivative of $Q_{\pi_\theta}(x,y)$ with respect to $y$ equals $G_\theta'(y)-q$ for $y<x$ and $G_\theta'(y)+k$ for $y>x$. At $y=x$, the derivative increases by $k+q\ge0$.

    Suppose first that $\underline\theta<\bar\theta$. Taking $u=\bar\theta$ in~\eqref{ineq: cash-lower-feasible-direction} and $u=\underline\theta$ in~\eqref{ineq: cash-upper-feasible-direction} gives
    \begin{equation*}
        k+G_\theta'(\underline\theta)\ge-\frac{\mathcal R(\theta)}{\alpha(1-\gamma)}, \qquad G_\theta'(\bar\theta)-q\le\frac{\mathcal R(\theta)}{\alpha(1-\gamma)}.
    \end{equation*}
    If $\underline\theta>\underline B$, decreasing the lower threshold in~\eqref{ineq: cash-lower-feasible-direction} also gives $k+G_\theta'(\underline\theta)\le\mathcal R(\theta)/[\alpha(1-\gamma)]$. If $\bar\theta<\bar B$, increasing the upper threshold in~\eqref{ineq: cash-upper-feasible-direction} gives $G_\theta'(\bar\theta)-q\ge-\mathcal R(\theta)/[\alpha(1-\gamma)]$. For actions above $a$, we use the lower bound at $\underline\theta$ when $x\le a$, and the lower bound at $\bar\theta=a$ when $x>a$. For actions below $a$, we use the upper bound at $\bar\theta$ when $x\ge a$, and the upper bound at $\underline\theta=a$ when $x<a$. The nonnegative derivative jump at $y=x$ preserves the required bounds when the action crosses the current state.

    When $\underline\theta=\bar\theta=a$, increasing the upper threshold, whenever $a<\bar B$, gives $G_\theta'(a)-q\ge-\mathcal R(\theta)/[\alpha(1-\gamma)]$. Since $k+q\ge0$, the same lower bound holds for $G_\theta'(a)+k$. Similarly, decreasing the lower threshold, whenever $a>\underline B$, gives $G_\theta'(a)+k\le\mathcal R(\theta)/[\alpha(1-\gamma)]$, and the same upper bound holds for $G_\theta'(a)-q$. These bounds cover every feasible direction from $a$.

    Combining the threshold derivative bounds with~\eqref{ineq: cash-G-slope-linear} yields
    \begin{equation}\label{ineq: cash-path-derivative-linear}
        \begin{aligned}
            \partial_-Q_{\pi_\theta}(x,a+t)&\ge-\frac{\mathcal R(\theta)}{\alpha(1-\gamma)^2}, &&0<t<\bar B-a,\\
            \partial_+Q_{\pi_\theta}(x,a-t)&\le\frac{\mathcal R(\theta)}{\alpha(1-\gamma)^2}, &&0<t<a-\underline B,
        \end{aligned}
    \end{equation}
    where the one-sided derivatives are taken with respect to the action. Integrating gives
    \begin{equation}\label{ineq: cash-one-step-linear}
        Q_{\pi_\theta}(x,y)-Q_{\pi_\theta}(x,\pi_\theta(x))\ge-\frac{\mathcal R(\theta)}{\alpha(1-\gamma)^2}|y-\pi_\theta(x)|, \qquad x\in\R,\ y\in[\underline B,\bar B].
    \end{equation}
    Since $|y-\pi_\theta(x)|\le\bar B-\underline B$, the Bellman equation implies
    \begin{equation*}
        \begin{aligned}
            J_{\pi_\theta}(x)-(TJ_{\pi_\theta})(x) =Q_{\pi_\theta}(x,\pi_\theta(x))-\min_{y\in[\underline B,\bar B]}Q_{\pi_\theta}(x,y) \le\frac{\bar B-\underline B}{\alpha(1-\gamma)^2}\mathcal R(\theta), \qquad \forall x\in\R.
        \end{aligned}
    \end{equation*}
    Since $\theta$ was arbitrary, Condition~\eqref{condition: general-mdp-pointwise-error} holds with exponent $1$ and $C=(\bar B-\underline B)/[\alpha(1-\gamma)^2]$. Theorem~\ref{thm: general-mdp-pointwise-PLK} gives~\eqref{ineq: cash-balance-PLK-linear}. At any stationary point $\theta$ of $l$ over $\Theta$, $\mathcal R(\theta)=0$, so~\eqref{ineq: cash-balance-PLK-linear} gives $l(\theta)=l(\theta^*)$.

    We next prove part~(ii). Suppose Assumption~\ref{assumption: cash balance}.\ref{assumption: cash balance-density-lower-bound} holds and $h+p>0$. Replacing~\eqref{ineq: cash-G-slope-linear} with~\eqref{ineq: cash-G-slope-quadratic} in the preceding derivative argument gives
    \begin{equation*}
        \begin{aligned}
            \partial_-Q_{\pi_\theta}(x,a+t)&\ge-\frac{\mathcal R(\theta)}{\alpha(1-\gamma)^2}+(h+p)\mu_Dt, &&0<t<\bar B-a,\\
            \partial_+Q_{\pi_\theta}(x,a-t)&\le\frac{\mathcal R(\theta)}{\alpha(1-\gamma)^2}-(h+p)\mu_Dt, &&0<t<a-\underline B.
        \end{aligned}
    \end{equation*}
    Integrating and completing the square yield, for every $x\in\R$ and $y\in[\underline B,\bar B]$,
    \begin{equation}\label{ineq: cash-one-step-quadratic}
        \begin{aligned}
            Q_{\pi_\theta}(x,y)-Q_{\pi_\theta}(x,\pi_\theta(x)) & \ge-\frac{\mathcal R(\theta)}{\alpha(1-\gamma)^2}|y-\pi_\theta(x)|+\frac{(h+p)\mu_D}{2}|y-\pi_\theta(x)|^2\\
            & \ge-\frac{\mathcal R(\theta)^2}{2(h+p)\mu_D\alpha^2(1-\gamma)^4}.
        \end{aligned}
    \end{equation}
    Consequently,
    \begin{equation*}
        J_{\pi_\theta}(x)-(TJ_{\pi_\theta})(x)\le\frac{\mathcal R(\theta)^2}{2(h+p)\mu_D\alpha^2(1-\gamma)^4}, \qquad \forall x\in\R.
    \end{equation*}
    Thus, Condition~\eqref{condition: general-mdp-pointwise-error} holds with exponent $2$ and $C=1/[2(h+p)\mu_D\alpha^2(1-\gamma)^4]$. Theorem~\ref{thm: general-mdp-pointwise-PLK} gives~\eqref{ineq: cash-balance-PLK}. \Halmos
\end{proof}

\begin{lemma}\label{lemma: cash-balance-smoothness}
    Suppose the first three parts of Assumption~\ref{assumption: cash balance} hold. Then $\nabla l$ is Lipschitz continuous on $\Theta$ with respect to the Euclidean norm.
\end{lemma}

\begin{proof}{Proof of Lemma~\ref{lemma: cash-balance-smoothness}}
    If $\underline B=\bar B$, then $\Theta$ is a singleton, and the conclusion is immediate. Suppose $\underline B<\bar B$. By Lemma~\ref{lemma: cash-balance-regularity}, $\nabla l$ is continuous on $\Theta$. Hence, it suffices to establish a uniform Lipschitz bound for $\theta,\theta'\in\operatorname{int}\Theta$. Couple the trajectories using the same initial state and demand sequence, and write $y_t(\theta)=\pi_\theta(x_t(\theta))$. The recursion gives
    \begin{equation*}
        |x_t(\theta)-x_t(\theta')|\le\|\theta-\theta'\|_2, \qquad |y_t(\theta)-y_t(\theta')|\le\|\theta-\theta'\|_2.
    \end{equation*}
    As in the proof of Lemma~\ref{lemma: markov-demand-smoothness}, the pathwise gradients of $x_t$ and $y_t$ are either zero or coordinate vectors almost surely. If the two trajectories agree on whether to adjust upward, make no adjustment, or adjust downward in every period through time $t$, their pathwise gradients agree. Otherwise, for some $s\le t$, $x_s(\theta)$ lies within $2\|\theta-\theta'\|_2$ of one of its two thresholds. Applying the same density and union-bound argument to both thresholds, the probability of at least one disagreement through time $t$ is at most
    \begin{equation*}
        8(L_\rho+tL_D)\|\theta-\theta'\|_2.
    \end{equation*}

    The transaction-cost derivative is $k(\nabla y_t(\theta)-\nabla x_t(\theta))$ under an upward adjustment, $q(\nabla x_t(\theta)-\nabla y_t(\theta))$ under a downward adjustment, and zero otherwise. Since $|C'(y)|\le\max\{h,p\}$, the period-cost derivative satisfies
    \begin{equation*}
        \left\|\nabla_\theta g(x_t(\theta),y_t(\theta))\right\|_2\le2\max\{|k|,|q|\}+\max\{h,p\}.
    \end{equation*}
    When the adjustment decisions agree through time $t$, the transaction-cost derivatives agree, and the $(h+p)L_D$-Lipschitz continuity of $C'$ bounds the difference between the period-cost derivatives by $(h+p)L_D\|\theta-\theta'\|_2$. On the remaining event, their difference is bounded by twice the preceding uniform bound. Differentiating the discounted cost series and summing these estimates as in the proof of Lemma~\ref{lemma: markov-demand-smoothness} yield
    \begin{equation*}
        \|\nabla l(\theta)-\nabla l(\theta')\|_2 \le\left[(h+p)L_D+16\left(2\max\{|k|,|q|\}+\max\{h,p\}\right)\left(L_\rho+\frac{\gamma L_D}{1-\gamma}\right)\right]\|\theta-\theta'\|_2.
    \end{equation*}
    Continuity of $\nabla l$ extends this bound to all $\theta,\theta'\in\Theta$. \Halmos
\end{proof}

\end{document}